\documentclass{amsart}

\usepackage[T1]{fontenc}
\usepackage{amsmath}
\usepackage{amssymb}
\usepackage{amsthm}
\usepackage{mathtools}
\usepackage[a4paper,margin=1in]{geometry}
\usepackage{hyperref}

\hypersetup{
  colorlinks=true,
  linkcolor=blue,
  citecolor=blue,
  urlcolor=blue
}
\allowdisplaybreaks
\theoremstyle{plain}

\newtheorem{lemma}{Lemma}
\newtheorem{theorem}{Theorem}
\newtheorem{corollary}{Corollary}
\newtheorem{proposition}{Proposition}

\theoremstyle{definition}
\newtheorem{definition}{Definition}
\theoremstyle{remark}
\newtheorem{remark}{Remark}

\newcommand{\1}{\mathbf 1}
\newcommand{\wt}{\widetilde}
\newcommand{\R}{\mathbb R}
\newcommand{\KL}{\operatorname{D}}

\newcommand{\E}{\mathbb E}

\newcommand{\Pbb}{\mathbb P}
\newcommand{\cN}{\mathcal N}

\newcommand{\Var}{\operatorname{Var}}

\DeclareMathOperator{\Ichi}{I_2}

\title{Gaussian Convexity Principles for Sharp Moderate Deviations of Gaussian Maxima and Critical SK Free-Energy Variance}
\author{Yiming Chen}
\address{School of Mathematical Sciences, Peking University}
\email{ymchenmath@math.pku.edu.cn}
\subjclass[2020]{Primary 60G15, 82B44; Secondary 60F10, 60K35, 94A17}
\keywords{Gaussian fields, moderate deviations, Gaussian maxima, Sherrington--Kirkpatrick model, free energy, entropy, information percolation}

\begin{document}

\begin{abstract}
In this paper, we establish a moderate deviation bound for Gaussian maxima and the variance asymptotics of the Sherrington–Kirkpatrick free energy at criticality based on Gaussian convexity.  First, let $(X_1,\ldots,X_N)$ be centered Gaussian vector with $\operatorname{Var}(X_i)\leq 1$. Suppose that, for fixed $\alpha\in(0,\sqrt 2)$ and $\kappa>0$, $\E\max_iX_i\geq\alpha\sqrt{\log N}$ and $\E\max_iX_i+\kappa\sqrt{\log N}\leq\sqrt{2\log N}$. We prove that
\[
\Pbb\left(\max_iX_i\geq \E\max_iX_i+\kappa\sqrt{\log N}\right)
\leq N^{-\kappa^2/(2-\alpha^2)+o(1)}.
\]
This answers a question of Ding, Eldan and Zhai \cite{DEZ2015}. The exponent is sharp, as witnessed by an equicorrelated Gaussian field. 

Second, for the Sherrington--Kirkpatrick model at the critical inverse temperature $\beta_c=1/\sqrt2$, we prove
\[
\Var\bigl(F_N(\beta_c)\bigr)=\frac16\log N+O(1).
\]

Our argument provides the variance asymptotics at the critical temperature from an entropy perspective, via a route distinct from that of Du and Huang~\cite{DuHuang2026}. For the upper bound, we express the variance as an entropy under exponential tilting and identify this entropy with the Kullback--Leibler divergence of a Gaussian synchronization model. Its derivative is then bounded using the I-MMSE formula, information percolation, and estimates for the susceptibility of the critical Erd\H{o}s--R'enyi random graph. For the lower bound, we combine Gaussian convexity applied at the replica parameter with an estimate for inverse moments on the sphere and an identity relating GOE eigenvalue densities in consecutive dimensions.

\end{abstract} 


\maketitle

\section{Introduction}\label{sec:introduction}
Convex functionals of Gaussian fields form a natural meeting point of Gaussian analysis, convex geometry, and disordered probability. For instance, Ehrhard’s inequality provides a fundamental tool for deriving lower tail bound and small-ball estimates for Gaussian processes and convex bodies \cite{Ehr1983,Borell2003,PaourisValettas2018}. Convexity also plays an important role in statistical-mechanical models, notably in the study of the free energies of spin glasses and directed polymers \cite{Panchenko2014,CarmonaHu2002, Chen2024GaussianConvexity}. 


In this paper, we use convexity principles for Gaussian field to obtain results on Gaussian maxima and critical SK free energy. Firstly, using the concavity of the Gaussian quantiles of convex sublevel sets implied by Ehrhard’s inequality, we derive a moderate deviation exponent in terms of the expected maximum. For the Sherrington--Kirkpatrick model, we give a new proof of the critical variance asymptotics based on the convexity of normalized logarithmic moment generating functions, which connects the variance to an entropy quantity and to moments at different replica parameters.

\textbf{Gaussian maxima.}
Let $\max_{1\leq i\leq N}X_i$ be centered Gaussian vector satisfying $\max_i\Var(X_i)\leq1$. The Borell--Sudakov--Tsirelson inequality gives
\begin{equation}\label{eq:BST}
\Pbb\bigl(\max_{1\leq i\leq N}X_i\geq \E \max_{1\leq i\leq N}X_i+t\bigr)\leq e^{-t^2/2},\qquad t\geq0.
\end{equation}
At the moderate deviation scale $t=\kappa\sqrt{\log N}$, \eqref{eq:BST} yields the exponent $\kappa^2/2$, independently of the correlations. Ding, Eldan and Zhai \cite{DEZ2015} showed that the exponent can be improved whenever $\E \max_{1\leq i\leq N}X_i$ has order $\sqrt{\log N}$, and asked for the optimal dependence on the size of the expected maximum.

The extremal example is an equicorrelated field. Let $G_0,G_1,\ldots,G_N$ be independent standard Gaussian variables and set
\begin{equation}\label{eq:equi}
X_i=\sqrt{1-\frac{\alpha^2}{2}}\,G_0+\frac{\alpha}{\sqrt2}G_i,
\qquad 1\leq i\leq N.
\end{equation}
Then $\E \max_{1\leq i\leq N}X_i=\alpha\sqrt{\log N}+o(\sqrt{\log N})$, while the common Gaussian component produces
\begin{equation}\label{eq:sharp_moderate_deviation}
\Pbb\left(\max_{1\leq i\leq N}X_i\geq \E \max_{1\leq i\leq N}X_i+\kappa\sqrt{\log N}\right)
=N^{-\kappa^2/(2-\alpha^2)+o(1)}.
\end{equation}
Our first result proves that the exponent in \eqref{eq:sharp_moderate_deviation} is universal under the natural condition appearing in the question of Ding, Eldan and Zhai.

\begin{theorem}
\label{thm:main}
Fix $\alpha\in(0,\sqrt2)$ and $\kappa>0$.
For each $N\ge2$, let $(X_1,\dots,X_N)$ be a centered Gaussian vector satisfying $\operatorname{Var}(X_i)\le 1$.
Assume that, for all sufficiently large $N$, $\E \max_{1\le i\le N}X_i \ge \alpha\sqrt{\log N},$ and $\E \max_{1\le i\le N}X_i + \kappa\sqrt{\log N} \le \sqrt{2\log N}.$ Then, as $N\to\infty$,
\begin{equation}\label{equ-main-1}
    \mathbb P\!\left(\max_{1\le i\le N}X_i\ge \E \max_{1\le i\le N} X_i+\kappa\sqrt{\log N}\right)
\le
N^{-\frac{\kappa^2}{2-\alpha^2}+o(1)}.
\end{equation}

\end{theorem}
Theorem~\ref{thm:main} is sharp in view of \eqref{eq:equi}--\eqref{eq:sharp_moderate_deviation}. The proof is based on Ehrhard's concavity theorem for the Gaussian quantile of the sublevel sets of \(\max_{1\leq i\leq N}X_i\). Combined with a union bound and two sided Mills ratio bound yield the exponent in \eqref{equ-main-1}.

\begin{remark}
Our use of Ehrhard's inequality is related to the argument of Paouris and Valettas \cite{PaourisValettas2018} for lower deviations of convex Gaussian functionals. Their proof bounds the transformed distribution function by its tangent line at a median, whereas our argument uses a chord joining the median to a remote deterministic level in order to derive the upper moderate deviation exponent.
\end{remark}


\textbf{Critical SK free energy.}
The Sherrington--Kirkpatrick model, introduced in \cite{SK1975}, is a basic mean-field model of spin glasses; see \cite{MPV1987,Pan2013,Tal2003,Tal2011spinI,Tal2011spinII}. For $\Sigma_N=\{-1,+1\}^N$, let $g=(g_{ij})_{1\leq i,j\leq N}$ be independent standard Gaussian variables and define
\[
H_N(\sigma;g)=\frac1{\sqrt N}\sum_{i,j=1}^Ng_{ij}\sigma_i\sigma_j,
\qquad \sigma\in\Sigma_N.
\]
Its covariance is
\[
\E\bigl[H_N(\sigma;g)H_N(\rho;g)\bigr]
=N R(\sigma,\rho)^2,
\]
where $R(\sigma,\rho)=\frac1N\sum_{i=1}^N\sigma_i\rho_i.$ For $\beta>0$, write $Z_N(\beta;g)=\sum_{\sigma\in\Sigma_N}e^{\beta H_N(\sigma;g)},$ $F_N(\beta;g)=\log Z_N(\beta;g),$ when the disorder is understood, we write $Z_N(\beta)$ and $F_N(\beta)$. All expectations and variances below are with respect to the disorder unless otherwise specified.

The critical inverse temperature is $\beta_c=1/\sqrt2$. For fixed $\beta<\beta_c$, the limiting variance diverges as $\beta\uparrow\beta_c$ like
\begin{equation}
\label{eq:variance_asymptotic}
-\frac{1}{2}\log\left(1-\frac{\beta^2}{\beta_c^2}\right)+O(1).
\end{equation}

The critical window has width $N^{-1/3}$ in $\beta_c^2-\beta^2$, so \eqref{eq:variance_asymptotic} predicts a cutoff of size $\frac16\log N$. Thus, the physics literature \cite{Asp2008,PR2009} predict
\begin{equation}\label{pre-ask}
\Var\bigl(F_N(\beta_c)\bigr)=\frac16\log N+O(1).
\end{equation}

Talagrand first obtained an upper bound of order $\sqrt N$ at criticality, which was later substantially sharpened by Chen and Lam as follows. 

\begin{theorem}[\cite{CL2019}]
  There exist a constant $C>0$ such that
   \begin{equation}
\operatorname{Var}(F_N(\beta_c))\le C(\log N)^2.
\label{eq:var_bound}
\end{equation}
\end{theorem}

Dey and Kang \cite{DK2026} proved the variance asymptotic $\frac16\log N+O(1)$, together with a central limit theorem, in the near-critical high temperature window $\beta_N^2=\beta_c^2-cN^{-1/3}+o(N^{-1/3})$ for fixed $c>0$. Schertzer \cite{Schertzer2026} obtained the first logarithmic critical upper bound and a diverging lower bound,
\[
\frac12\log\log\log N-O(1)
\leq \Var\bigl(F_N(\beta_c)\bigr)
\leq \frac14\log N+O(1).
\]
More recently, Du and Huang \cite{DuHuang2026} proved the full critical variance asymptotic \eqref{pre-ask}, a Gaussian central limit theorem, and sharp bounds for the two-replica overlap at criticality. Their proof relies on spherical comparison, critical random matrix estimates, and the cavity method.

Our second main result provides a different proof for the variance asymptotic, organized around entropy and Gaussian convexity.

\begin{theorem}\label{thm:improved-critical}
For every integer $N\ge 1$,
\[
\Var\bigl(F_N(\beta_c)\bigr)= \frac16\log N + O(1).
\]
\end{theorem}

For the upper bound, Gaussian convexity first gives $\Var(F_N)\leq2B_N$, where $B_N$ is an entropy under exponential tilting. After separating the diagonal disorder, $B_N$ is identified with the Kullback--Leibler divergence between the planted and null laws of a Gaussian synchronization model. The I-MMSE formula expresses its derivative through squared posterior edge correlations. An information-percolation inequality from Abbe--Boix-Adser~\cite{AbbeBoix2020} then bounds these correlations by connection probabilities in an Erd\H{o}s--R\'enyi graph in critical window, and the critical susceptibility estimate supplies the $O(1)$ cost of moving from $1-N^{-1/3}$ to the critical point.

For the lower bound, the Gaussian convexity principle from Wei-Kuo Chen~ \cite{Chen2024GaussianConvexity} is used to reduce the variance to an inverse second moment of the normalized partition function and the entropy estimated in the upper bound. Haar averaging compares the inverse moment with its spherical counterpart. Then a one replica contour formula, GOE soft edge information, and a dimension shift identity for the eigenvalue density yield the $N^{1/2}$. 

The rest of the paper is organized as follows. Section~\ref{sec:gaussian-maxima} proves Theorem~\ref{thm:main}. Section~\ref{sec:improved-critical-variance} proves Theorem~\ref{thm:improved-critical}: the entropy upper bound is established in Section~\ref{subsec:critical-upper-bound}, and the lower bound is proved in Section~\ref{subsec:critical-lower-bound}.

\section{Moderate deviations for Gaussian maxima}\label{sec:gaussian-maxima}




Let $\Phi$ denote the distribution function of the Gaussian law. For $t \in \mathbb{R}$ and $u\in(0,1)$, 
write $\bar\Phi(t)=1-\Phi(t),$ $\bar\Phi^{-1}(u)=\Phi^{-1}(1-u)$.  Thus 
$$
\mathbb P(\max_{1\le i\le N}X_i> t)=\bar\Phi(\Phi^{-1}\!\big(\mathbb P(\max_{1\le i\le N}X_i\le t)\big).
$$

Note that the function $\bar\Phi(\cdot)$ is decreasing. Our goal is to derive a lower bound for $\Phi^{-1}\!\left( \mathbb P\left(\max_{1\le i\le N}X_i\le t\right) \right),$ and therefore obtain an upper bound for \eqref{equ-main-1}. First, we need to introduce some lemmas.

\begin{lemma}[\cite{Ehr1983}]\label{lem:concavity}
The function $\Phi^{-1}\!\left( \mathbb P\left(\max_{1\le i\le N}X_i\le t\right) \right)$ is concave on $\mathbb{R}$.
\end{lemma}

Let $m_N$ be a median of $\max_{1\le i\le N}X_i$, that is, $\mathbb P\!\left(\max_{1\le i\le N}X_i\le m_N\right)\ge \frac12,$ and $\mathbb P\!\left(\max_{1\le i\le N}X_i\ge m_N\right)\ge \frac12.$ Since $\Phi^{-1}$ is increasing and $\Phi^{-1}(1/2)=0$, it follows that
\begin{equation}\label{medde}
\Phi^{-1}\!\left(
    \mathbb P\left(\max_{1\le i\le N}X_i\le m_N\right)
\right)
\ge 0.
\end{equation}


By the Gaussian concentration inequality around a median, we have

\begin{lemma}[\cite{VER2020}]\label{lem:median-mean}
For every $u\ge 0$,
\[
\mathbb P\bigl(|\max_{1\le i\le N}X_i-m_N|\ge u\bigr)\le 2e^{-u^2/2}.
\]
Consequently,
\begin{equation}\label{mean-differ-median}
    | \E \max_{1\le i\le N} X_i-m_N|
\le \E|\max_{1\le i\le N}X_i-m_N|
\le \int_0^\infty 2e^{-u^2/2}\,du
=\sqrt{2\pi}=C_{\mathrm{med}}.
\end{equation}

\end{lemma}

Hence, for all sufficiently large $N$, 
\begin{equation}\label{compu rnbn}
     \kappa \sqrt{\log N} + \E \max_{1\le i\le N} X_i - m_N
    \ge \kappa \sqrt{\log N} - C_{\mathrm{med}}
    \ge 0.
\end{equation}

Moreover, for every fixed $t>\sqrt2$, we have

\begin{equation}\label{compare-equ}
    b_N\le \kappa \sqrt{\log N} + \E \max_{1\le i\le N} X_i\le t\sqrt{\log N}.
\end{equation}

Now we are ready to derive the lower bound of $\Phi^{-1}\left(
\mathbb P\left(
\max_{1\le i\le N}X_i
\le
\kappa\sqrt{\log N}
+
\mathbb E\max_{1\le i\le N}X_i
\right)
\right)$ based on the concavity of $\Phi^{-1}\left(
\mathbb P\left(
\max_{1\le i\le N}X_i\le t
\right)
\right)$ and \eqref{compare-equ}.
\begin{lemma}\label{lem:psi-lower-bound-at-rN} Let \(t>\sqrt{2}\) be fixed. Then \begin{align} &\Phi^{-1}\!\left( \mathbb P\left( \max_{1\le i\le N}X_i \le \kappa\sqrt{\log N} + \mathbb E\max_{1\le i\le N}X_i \right) \right) \notag\\ &\qquad\ge \left( \frac{\kappa}{t-\alpha}+o(1) \right) \Phi^{-1}\!\left( \mathbb P\left( \max_{1\le i\le N}X_i \le t\sqrt{\log N} \right) \right), \label{eq:concavity-lower-bound} \end{align} 
which implies \begin{align} &\mathbb P\left( \max_{1\le i\le N}X_i \ge \kappa\sqrt{\log N} + \mathbb E\max_{1\le i\le N}X_i \right) \notag\\ &\qquad\le \exp\left\{ - \left( \frac{\kappa}{t-\alpha}+o(1) \right)^2 \left[ \Phi^{-1}\!\left( \mathbb P\left( \max_{1\le i\le N}X_i \le t\sqrt{\log N} \right) \right) \right]^2 \right\}. \end{align} \end{lemma}

\begin{proof} Fix \(t>\sqrt{2}\). By \eqref{compare-equ}, for all sufficiently large \(N\), \[ m_N \le \kappa\sqrt{\log N} + \mathbb E\max_{1\le i\le N}X_i \le t\sqrt{\log N}. \] Set \[ \lambda_N = \frac{ \kappa\sqrt{\log N} + \mathbb E\max_{1\le i\le N}X_i - m_N }{ t\sqrt{\log N} - m_N } \in[0,1]. \] Then \[ \kappa\sqrt{\log N} + \mathbb E\max_{1\le i\le N}X_i = \lambda_N t\sqrt{\log N} + (1-\lambda_N)m_N . \] By the concavity in Lemma~\ref{lem:concavity}, 
\begin{equation}\label{eq:concavity-interpolation}
\begin{aligned}
    &\Phi^{-1}\!\left(
        \mathbb P\left(
            \max_{1\le i\le N}X_i
            \le
            \kappa\sqrt{\log N}
            +
            \mathbb E\max_{1\le i\le N}X_i
        \right)
    \right)                                      \\
    &\qquad\ge
    \lambda_N
    \Phi^{-1}\!\left(
        \mathbb P\left(
            \max_{1\le i\le N}X_i
            \le
            t\sqrt{\log N}
        \right)
    \right)                                      \\
    &\qquad\quad+
    (1-\lambda_N)
    \Phi^{-1}\!\left(
        \mathbb P\left(
            \max_{1\le i\le N}X_i
            \le
            m_N
        \right)
    \right).
\end{aligned}
\end{equation}

 Since \[ \Phi^{-1}\!\left( \mathbb P\left( \max_{1\le i\le N}X_i \le m_N \right) \right) \ge 0 \] and \(1-\lambda_N\ge0\), the second term on the right-hand side of \eqref{eq:concavity-interpolation} is nonnegative. Therefore, \begin{align} &\Phi^{-1}\!\left( \mathbb P\left( \max_{1\le i\le N}X_i \le \kappa\sqrt{\log N} + \mathbb E\max_{1\le i\le N}X_i \right) \right) \notag\\ &\qquad\ge \frac{ \kappa\sqrt{\log N} + \mathbb E\max_{1\le i\le N}X_i - m_N }{ t\sqrt{\log N} - m_N } \Phi^{-1}\!\left( \mathbb P\left( \max_{1\le i\le N}X_i \le t\sqrt{\log N} \right) \right). \label{eq:concavity-step} \end{align} 
Using $\left| \mathbb E\max_{1\le i\le N}X_i - m_N \right| \le C_{\mathrm{med}}$ and $\mathbb E\max_{1\le i\le N}X_i \ge \alpha\sqrt{\log N},$ we estimate \begin{align} &\frac{ \kappa\sqrt{\log N} + \mathbb E\max_{1\le i\le N}X_i - m_N }{ t\sqrt{\log N} - m_N } \notag\\ &\qquad= \frac{ \kappa\sqrt{\log N} + \left( \mathbb E\max_{1\le i\le N}X_i - m_N \right) }{ t\sqrt{\log N} - \mathbb E\max_{1\le i\le N}X_i + \left( \mathbb E\max_{1\le i\le N}X_i - m_N \right) } \notag\\ &\qquad\ge \frac{ \kappa\sqrt{\log N} - C_{\mathrm{med}} }{ (t-\alpha)\sqrt{\log N} + C_{\mathrm{med}} } \notag\\ &\qquad= \frac{\kappa+o(1)}{t-\alpha+o(1)} = \frac{\kappa}{t-\alpha}+o(1). \label{eq:lambda-lower-bound} \end{align} Combining \eqref{eq:concavity-step} and \eqref{eq:lambda-lower-bound}, we obtain \eqref{eq:concavity-lower-bound}. \end{proof}

Then we introduce a quantitative lower bound for $\Phi^{-1}\left(
\mathbb P\left(\max_{1\le i\le N}X_i\le t\sqrt{\log N}\right)
\right)$ using union bound and Gaussian tail bound.

\begin{lemma}\label{lem:psi-lower-bound-at-tscale}
Let $t>\sqrt{2}$ be fixed, and define
\[
    q_t=\sqrt{(t^2-2)\log N-2\log\log N}.
\]
Then, for all sufficiently large $N$,
\begin{equation}\label{eq:qt_lower_bound}
    \Phi^{-1}\left(
    \mathbb P\left(\max_{1\le i\le N}X_i\le t\sqrt{\log N}\right)
    \right)
    \ge q_t .
\end{equation}
In particular,
\begin{equation}\label{eq:qt_asymptotic}
    \Phi^{-1}\left(
    \mathbb P\left(\max_{1\le i\le N}X_i\le t\sqrt{\log N}\right)
    \right)
    =
    \sqrt{t^2-2}\sqrt{\log N}+o(\sqrt{\log N}),
    \qquad N\to\infty .
\end{equation}
\end{lemma}

\begin{proof}
Since $\operatorname{Var}(X_i)\le 1$, for every $1\le i\le N$ and $x>0$, we have $\mathbb P(X_i>x)\le \bar\Phi(x).$ Therefore, by the union bound,
\[
    \mathbb P\left(\max_{1\le i\le N}X_i>x\right)
    \le N\bar\Phi(x).
\]
Using $\bar\Phi(x)\le \frac{1}{x\sqrt{2\pi}}e^{-x^2/2},$ we obtain
\begin{equation}\label{eq:union_tail_bound}
    \mathbb P\left(\max_{1\le i\le N}X_i>t\sqrt{\log N}\right)
    \le
    \frac{N}{t\sqrt{\log N}\sqrt{2\pi}}
    e^{-t^2\log N/2}
    =
    \frac{1}{t\sqrt{\log N}\sqrt{2\pi}}N^{1-t^2/2}.
\end{equation}
Moreover,
\begin{align}
    \bar\Phi\left(
    \Phi^{-1}\left(
    \mathbb P\left(\max_{1\le i\le N}X_i\le x\right)
    \right)
    \right)
    &=
    1-\mathbb P\left(\max_{1\le i\le N}X_i\le x\right)  \notag\\
    &=
    \mathbb P\left(\max_{1\le i\le N}X_i>x\right).
\end{align}
Hence
\[
    \Phi^{-1}\left(
    \mathbb P\left(\max_{1\le i\le N}X_i\le t\sqrt{\log N}\right)
    \right)
    =
    \bar\Phi^{-1}\left(
    \mathbb P\left(\max_{1\le i\le N}X_i>t\sqrt{\log N}\right)
    \right).
\]

Recall that $q_t=\sqrt{(t^2-2)\log N-2\log\log N}.$ By the lower Mills ratio bound,
\begin{equation}\label{eq:lower_mills_ratio}
    \bar\Phi(x)
    \ge
    \frac{1}{\sqrt{2\pi}}\frac{x}{1+x^2}e^{-x^2/2}.
\end{equation}
Applying \eqref{eq:lower_mills_ratio} with $x=q_t$, and noting that
$q_t\asymp \sqrt{\log N}$, we get
\[
    \frac{q_t}{1+q_t^2}\asymp (\log N)^{-1/2},
\]
while
\[
    e^{-q_t^2/2}
    =
    e^{-\frac12[(t^2-2)\log N-2\log\log N]}
    =
    N^{1-t^2/2}\log N.
\]
Therefore,
\[
    \bar\Phi(q_t)
    \ge
    c(\log N)^{1/2}N^{1-t^2/2}
\]
for some absolute constant $c>0$. Hence, \eqref{eq:union_tail_bound} yields
\[
    \mathbb P\left(\max_{1\le i\le N}X_i>t\sqrt{\log N}\right)
    \le
    \frac{C}{\sqrt{\log N}}N^{1-t^2/2}
\]
for some absolute constant $C>0$, whereas
\[
    \bar\Phi(q_t)
    \ge
    c(\log N)^{1/2}N^{1-t^2/2}.
\]
Since $(\log N)^{1/2}\gg(\log N)^{-1/2}$ as $N\to\infty$, it follows that
\[
    \mathbb P\left(\max_{1\le i\le N}X_i>t\sqrt{\log N}\right)
    \le
    \bar\Phi(q_t)
\]
for all sufficiently large $N$. Hence
\[
    \Phi^{-1}\left(
    \mathbb P\left(\max_{1\le i\le N}X_i\le t\sqrt{\log N}\right)
    \right)
    \ge q_t,
\]
which proves \eqref{eq:qt_lower_bound}. The asymptotic statement \eqref{eq:qt_asymptotic} follows from the definition of $q_t$.
\end{proof}

Combining Lemma~\ref{lem:psi-lower-bound-at-rN} with Lemma~\ref{lem:psi-lower-bound-at-tscale}, we obtain the following quantitative estimate.

\begin{proposition}\label{prop:tail-bound-fixed-t}
For every fixed $t>\sqrt{2}$, we have
\begin{align}
    &\mathbb P\left(
        \max_{1\le i\le N}X_i
        >
        \kappa\sqrt{\log N}
        +
        \mathbb E\max_{1\le i\le N}X_i
    \right) \notag\\
    &\qquad =
    \bar\Phi\left(
    \Phi^{-1}\left(
    \mathbb P\left(
        \max_{1\le i\le N}X_i
        \le
        \kappa\sqrt{\log N}
        +
        \mathbb E\max_{1\le i\le N}X_i
    \right)
    \right)
    \right) \notag\\
    &\qquad \le
    \exp\left(
    -\left(
    \frac{\kappa^2(t^2-2)}{2(t-\alpha)^2}
    +o(1)
    \right)\log N
    \right).
    \label{eq:proposition_quantitative_estimate}
\end{align}
\end{proposition}

\begin{proof}
Combining Lemma~3 and Lemma~4, we obtain
\begin{align}
    &\Phi^{-1}\left(
    \mathbb P\left(
        \max_{1\le i\le N}X_i
        \le
        \kappa\sqrt{\log N}
        +
        \mathbb E\max_{1\le i\le N}X_i
    \right)
    \right) \notag\ge
    \left(\frac{\kappa}{t-\alpha}+o(1)\right)
    \Phi^{-1}\left(
    \mathbb P\left(
        \max_{1\le i\le N}X_i
        \le
        t\sqrt{\log N}
    \right)
    \right) \notag\\
    &\qquad\ge
    \left(
    \frac{\kappa\sqrt{t^2-2}}{t-\alpha}
    +o(1)
    \right)\sqrt{\log N}.   \label{eq:quantile_lower_bound_shifted_max}
\end{align}
In particular,
\[
    \Phi^{-1}\left(
    \mathbb P\left(
        \max_{1\le i\le N}X_i
        \le
        \kappa\sqrt{\log N}
        +
        \mathbb E\max_{1\le i\le N}X_i
    \right)
    \right)
    \to\infty.
\]
Using \eqref{eq:quantile_lower_bound_shifted_max} and the fact that $\bar\Phi(x)\le e^{-x^2/2},$ we obtain
\begin{align*}
    &\mathbb P\left(
        \max_{1\le i\le N}X_i
        >
        \kappa\sqrt{\log N}
        +
        \mathbb E\max_{1\le i\le N}X_i
    \right) \\
    &\qquad =
    \bar\Phi\left(
    \Phi^{-1}\left(
    \mathbb P\left(
        \max_{1\le i\le N}X_i
        \le
        \kappa\sqrt{\log N}
        +
        \mathbb E\max_{1\le i\le N}X_i
    \right)
    \right)
    \right) \\
    &\qquad \le
    \exp\left[
    -\frac12
    \left(
    \Phi^{-1}\left(
    \mathbb P\left(
        \max_{1\le i\le N}X_i
        \le
        \kappa\sqrt{\log N}
        +
        \mathbb E\max_{1\le i\le N}X_i
    \right)
    \right)
    \right)^2
    \right] \\
    &\qquad \le
    \exp\left(
    -\left(
    \frac{\kappa^2(t^2-2)}{2(t-\alpha)^2}
    +o(1)
    \right)\log N
    \right).
\end{align*}
\end{proof}

\begin{proof}[Proof of Theorem~1]
By Proposition~\ref{prop:tail-bound-fixed-t}, for every fixed $t>\sqrt{2}$,
\begin{align}
    &\mathbb P\left(
        \max_{1\le i\le N}X_i
        >
        \mathbb E\max_{1\le i\le N}X_i
        +
        \kappa\sqrt{\log N}
    \right) \notag\\
    &\qquad\le
    \exp\left(
    -\left(
    \frac{\kappa^2}{2}\frac{t^2-2}{(t-\alpha)^2}
    +o(1)
    \right)\log N
    \right).
    \label{eq:proof_thm1_pre_optimization}
\end{align}
It remains to optimize over $t>\sqrt{2}$. Set $h(t)=\frac{t^2-2}{(t-\alpha)^2}.$ A direct computation gives
\[
    h'(t)=\frac{2(2-\alpha t)}{(t-\alpha)^3}.
\]
Since $\alpha\in(0,\sqrt{2})$, we have $2/\alpha>\sqrt{2}$, and therefore $h$ attains its maximum on $(\sqrt{2},\infty)$ at $t_*=\frac{2}{\alpha}.$ Substituting $t_*$ yields $h(t_*)=\frac{2}{2-\alpha^2}.$ Hence, taking $t=t_*$ in \eqref{eq:proof_thm1_pre_optimization}, we obtain
\[
    \mathbb P\left(
        \max_{1\le i\le N}X_i
        >
        \mathbb E\max_{1\le i\le N}X_i
        +
        \kappa\sqrt{\log N}
    \right)
    \le
    N^{-\frac{\kappa^2}{2-\alpha^2}+o(1)}.
\]
Since $\max_{1\le i\le N}X_i$ has a continuous distribution,
\[
    \mathbb P\left(
        \max_{1\le i\le N}X_i
        \ge
        \mathbb E\max_{1\le i\le N}X_i
        +
        \kappa\sqrt{\log N}
    \right)
    =
    \mathbb P\left(
        \max_{1\le i\le N}X_i
        >
        \mathbb E\max_{1\le i\le N}X_i
        +
        \kappa\sqrt{\log N}
    \right).
\]

\end{proof}

\section{Critical SK free energy fluctuations}\label{sec:improved-critical-variance}


\subsection{Proof of upper bound}\label{subsec:critical-upper-bound}

\medskip
\noindent

Before the proof of upper bound of Theorem \ref{thm:improved-critical}, we introduce some preliminary. We first recall a Gaussian convexity principle of Wei-Kuo Chen that relates the variance to an entropy quantity arising from exponential tilting.

\begin{theorem}[Wei-Kuo Chen, {\cite[Theorem~1]{Chen2024GaussianConvexity}}]\label{thm:Chen}
Let $G$ be a standard Gaussian vector in $\R^m$, and let $\Psi:\R^m\to\R$ be convex. Assume that for every $s>0$, $\E e^{s\Psi(G)}<\infty$. Define
\[
\mathcal K_\Psi(s)=
\begin{cases}
\dfrac1s\log \E e^{s\Psi(G)}, & s \neq 0,\\[1ex]
\E\Psi(G), & s=0.
\end{cases}
\]
Then $\mathcal K$ is convex on $\mathbb{R}$.
\end{theorem}

The next result is information-theoretic. Once the entropy quantity is represented as a Kullback--Leibler divergence along a Gaussian channel, its derivative can be computed using the I--MMSE formula.

\begin{theorem}[I-MMSE formula, Guo, Shamai and Verdú\cite{GuoShamaiVerdu2005}]\label{thm:IMSE}
Let $U$ be an $m$-dimensional random vector with $\E\|U\|^2<\infty$, let $G\sim \cN(0,I_m)$ be independent of $U$, and for $\gamma\ge 0$ set
\[
Y^{\gamma}=\sqrt\gamma\,U+G.
\]
Then the mutual information (measured in nats) satisfies
\[
\frac{d}{d\gamma} I(U;Y^{\gamma})=\frac12\,\E\bigl[\|U-\E[U\mid Y^{\gamma}]\|^2\bigr].
\]
\end{theorem}



We also need an information percolation comparison that bounds conditional $\chi^2$-information by connectivity in a bond percolation model from Abbe and Boix-Adserà. Recall the conditional $\chi^2$-mutual information: if $(A,B,Y)$ are random variables, then
\[
\Ichi(A;B\mid Y)=\E\Bigl[\chi^2\bigl(P_{A,B\mid Y}, P_{A\mid Y}P_{B\mid Y}\bigr)\Bigr].
\]

\begin{definition}[Definition 3.5\cite{AbbeBoix2020}] Let \(\mathcal{G}=(V,E)\) be a finite graph. For each edge \(e=\{i,j\}\in E\), let \(Q_e(\cdot\mid\cdot)\)   be a probability kernel from the binary input alphabet \(\{-1,+1\}\) to a measurable output space \((\mathsf Y_e,\mathcal A_e)\). We write $Q_{e,+}(\cdot)=Q_e(\cdot\mid +1),  Q_{e,-}(\cdot)=Q_e(\cdot\mid -1).$ The associated graphical channel is the channel from \(\{-1,+1\}^V\) to \(\prod_{e\in E}\mathsf Y_e\) defined as follows: given \(x=(x_v)_{v\in V}\in\{-1,+1\}^V\), the edge observations \((Y_e)_{e\in E}\) are conditionally independent, and for \(e=\{i,j\}\), $Y_e \sim Q_e(\cdot\mid x_i x_j).$ We say that the edge channel \(Q_e\) is symmetric if there exists a measurable involution $T_e:\mathsf Y_e\to \mathsf Y_e, T_e^{-1}=T_e,$ such that for every \(A\in\mathcal A_e\), \[ Q_e(A\mid +1)=Q_e(T_e(A)\mid -1). \] Equivalently, since \(T_e\) is an involution, \[ Q_e(T_e(A)\mid +1)=Q_e(A\mid -1), \qquad A\in\mathcal A_e. \] The graphical channel is called symmetric if \(Q_e\) is symmetric for every \(e\in E\). \end{definition}

\begin{theorem}[Theorem~3.6, Abbe--Boix-Adser, \cite{AbbeBoix2020}]\label{thm:AB}
Consider a binary synchronization model on a finite graph \(\mathcal{G}=(V,E)\): the hidden vertex labels \((\theta_v)_{v\in V}\) are i.i.d. uniform on \(\{-1,+1\}\), and conditional on \(\theta\), the edge observations \((Y_e)_{e\in E}\) are generated by a symmetric graphical channel. Then, for every \(u,v\in V\), \[ I_2(\theta_u;\theta_v\mid Y) \le \mathbb P\bigl( u \leftrightarrow v \text{ in bond percolation on }\mathcal{G} \text{ with edge-open probabilities }(p_e)_{e\in E} \bigr), \] where $p_e=I_2(\theta_i;\theta_j\mid Y_e), e=\{i,j\}\in E. $ Here \(I_2\) denotes \(\chi^2\)-mutual information, and the bond percolation on \(\mathcal{G}\) is the independent percolation process in which each edge \(e\) is open with probability \(p_e\).
\end{theorem}

The main idea of the proof can be summarized as follows. We first convert the upper bound of variance of $F_N(\beta_c)$ into an entropy quantity $B_N(\beta_c)$, and then, using a Hamiltonian decomposition to obatin a diagonal contribution plus a SK off-diagonal contribution, $\widetilde{B}_N(1)$. Subsequently, the I-MMSE relation in Theorem \ref{thm:IMSE} is employed to express the derivative of $\widetilde{B}_N(\lambda)$ as the squared edge correlation. By invoking the Abbe-Boix percolation bounds in Theorem \ref{thm:AB}, it is shown that this derivative is at most $O(N^{1/3})$ within the critical window. Therefore, integrating from $\lambda_0 = 1 - N^{-1/3}$ to $1$ incurs only an $O(1)$ loss, ultimately yielding $\operatorname{Var}(F_N(\beta_c)) \le \frac{1}{6}\log N + O(1)$.

For fixed $\beta>0$, define $\Lambda_{N,\beta}(s)=\log \E e^{sF_N(\beta)},$
and then $$\mathcal K_{N,\beta}(s)=
\begin{cases}
\dfrac{\Lambda_{N,\beta}(s)}{s}, & s>0,\\[1ex]
\E F_N(\beta), & s=0.
\end{cases}$$

Then we ensure that the $\mathcal K_{N,\beta}$ is convex on \([0,\infty)\) based on Theorem \ref{thm:Chen}.


\begin{lemma}\label{lem:FNconvex}
For every $N\ge 1$, $s\in\R$ and $\beta>0$, the map $g\mapsto F_N(\beta;g)$ is convex on $\R^{N^2}$ and satisfies $\E e^{sF_N(\beta)}<\infty.$ Consequently, $\mathcal K_{N,\beta}$ is convex on $[0,\infty)$.
\end{lemma}

\begin{proof}
For each fixed $\sigma\in\Sigma_N$, the map $g\mapsto \frac{\beta}{\sqrt N}\sum_{i,j=1}^N g_{ij}\sigma_i\sigma_j$ is affine, hence 
\[
F_N(\beta;g)=\log\sum_{\sigma\in\Sigma_N}\exp\!\Bigl(\frac{\beta}{\sqrt N}\sum_{i,j=1}^N g_{ij}\sigma_i\sigma_j\Bigr)
\]
is convex. Moreover,
\[
F_N(\beta;g)
\le N\log 2 + \frac{\beta}{\sqrt N}\sum_{i,j=1}^N |g_{ij}|,
\]
therefore, for every $s>0$,
\[
\E e^{sF_N(\beta)}
\le e^{sN\log 2}\prod_{i,j=1}^N \E e^{s\beta |g_{ij}|/\sqrt N}
<\infty.
\]
The same upper bound also yields finiteness for $s<0$ after applying it to $-F_N(\beta)$.


\end{proof}

Applying the Lemma~\ref{lem:FNconvex} to $F_N(\beta)$ yields the following entropy bound.

%

\begin{lemma}\label{prop:forwardKL}
For every $N\ge 1$ and every $\beta>0$,
\[
\Var\bigl(F_N(\beta)\bigr)\le\frac{\E\!\left[ Z_N(\beta) F_N(\beta) \right]}
         {\E Z_N(\beta)}
    -
    \log \E Z_N(\beta):= 2B_N(\beta),
\]

\end{lemma}

\begin{proof}
Since $F_N(\beta)$ has exponential moments in a neighborhood of the origin, $\Lambda_{N,\beta}$ is near $0$, and note that $\Lambda_{N,\beta}(0)=0,$ $\Lambda_{N,\beta}'(0)=\E F_N(\beta),$ $\Lambda_{N,\beta}''(0)=\Var\bigl(F_N(\beta)\bigr).$
 
 Therefore, as $s\downarrow 0$,
\[
\Lambda_{N,\beta}(s)=s\E F_N(\beta)+\frac{s^2}{2}\Var\bigl(F_N(\beta)\bigr)+o(s^2),
\]
whence
\[
\mathcal K_{N,\beta}(s)=\E F_N(\beta)+\frac{s}{2}\Var\bigl(F_N(\beta)\bigr)+o(s),
\]
and thus
\begin{equation}
\begin{aligned}
\mathcal{K}_{N,\beta}'(0+) 
&= \lim_{s\downarrow 0} \frac{\mathcal{K}_{N,\beta}(s) - \mathcal{K}_{N,\beta}(0)}{s} \\
&= \lim_{s\downarrow 0} \frac{\left[\mathbb{E}F_N(\beta) + \frac{s}{2}\Var(F_N(\beta)) + o(s)\right] - \mathbb{E}F_N(\beta)}{s} \\
&= \frac{1}{2}\Var(F_N(\beta)).
\end{aligned}
\label{eq:phi0prime}
\end{equation}
Since $\mathcal K_{N,\beta}$ is convex on $[0,\infty)$, its derivative is monotone increasing, so
\[
\mathcal K_{N,\beta}'(0+)\le \mathcal K_{N,\beta}'(1).
\]
Note that $\mathcal K_{N,\beta}(s)=\Lambda_{N,\beta}(s)/s$ for $s>0$,
\[
\mathcal K_{N,\beta}'(1)=\Lambda_{N,\beta}'(1)-\Lambda_{N,\beta}(1).
\]
Using \eqref{eq:phi0prime}, 
\[
\frac12\Var\bigl(F_N(\beta)\bigr)
\le \Lambda_{N,\beta}'(1)-\Lambda_{N,\beta}(1).
\]
Finally,
\[
\Lambda_{N,\beta}(1)=\log \E e^{F_N(\beta)}=\log \E Z_N(\beta)
\]
and
\[
\Lambda_{N,\beta}'(1)=\frac{\E[F_N(\beta)e^{F_N(\beta)}]}{\E e^{F_N(\beta)}}
=\frac{\E[Z_N(\beta)F_N(\beta)]}{\E Z_N(\beta)}.
\]
\end{proof}

We decompose the representation of the SK Hamiltonian into its diagonal and off-diagonal parts. The diagonal term is independent of $\sigma$ and therefore contributes only a random shift common to all configurations. By combining the two couplings associated with each unordered pair, the off-diagonal part can be rewritten in the standard SK form with independent Gaussian couplings.


For $1\le i<j\le N$, let $\Delta_N=\frac1{\sqrt N}\sum_{i=1}^N g_{ii},$ $J_{ij}=\frac{g_{ij}+g_{ji}}{\sqrt2},$ and
\[
\wt H_N(\sigma;J)=\frac1{\sqrt N}\sum_{1\le i<j\le N} J_{ij}\sigma_i\sigma_j.
\]
Then $\Delta_N\sim \cN(0,1)$, the family $(J_{ij})_{1\le i<j\le N}$ consists of i.i.d. standard Gaussians, and it is independent of $\Delta_N$. Moreover,
\begin{equation}\label{eq:Hdecomp}
H_N(\sigma;g)=\Delta_N + \sqrt2\,\wt H_N(\sigma;J).
\end{equation}
Indeed,
\[
\sum_{i,j=1}^N g_{ij}\sigma_i\sigma_j
=\sum_{i=1}^N g_{ii} + \sum_{1\le i<j\le N} (g_{ij}+g_{ji})\sigma_i\sigma_j.
\]

Define the corresponding partition function and free energy $\wt Z_N(\lambda;J)=\sum_{\sigma\in\Sigma_N} e^{\lambda \wt H_N(\sigma;J)},$ $\wt F_N(\lambda;J)=\log \wt Z_N(\lambda;J).$ Then from \eqref{eq:Hdecomp}, with $\lambda=\sqrt2\,\beta$,
\begin{equation}\label{eq:Zdecomp}
Z_N(\beta;g)=e^{\beta\Delta_N}\wt Z_N(\lambda;J),
\qquad
F_N(\beta;g)=\beta\Delta_N + \wt F_N(\lambda;J).
\end{equation}
Define
\begin{equation}\label{eq:wtBdef}
\wt B_N(\lambda)=\frac{\E[\wt Z_N(\lambda)\wt F_N(\lambda)]}{\E \wt Z_N(\lambda)} - \log \E\wt Z_N(\lambda).
\end{equation}

\begin{lemma}\label{prop:Bdecomp}
For every $\beta>0$, with $\lambda=\sqrt2\,\beta$,
\[
B_N(\beta)=\frac{\beta^2}{2}+\wt B_N(\lambda).
\]
In particular,
\[
B_N(\beta_c)=\frac14 + \wt B_N(1).
\]

\end{lemma}

\begin{proof}
Using \eqref{eq:Zdecomp} and the independence of $\Delta_N$ and $J$,
\begin{align*}
\frac{\E[Z_N(\beta)F_N(\beta)]}{\E Z_N(\beta)}
&=
\frac{\E\bigl[e^{\beta\Delta_N}\wt Z_N(\lambda)(\beta\Delta_N+\wt F_N(\lambda))\bigr]}
{\E e^{\beta\Delta_N}\, \E \wt Z_N(\lambda)}\\
&=
\beta\,\frac{\E[\Delta_N e^{\beta\Delta_N}]}{\E e^{\beta\Delta_N}}
+
\frac{\E[\wt Z_N(\lambda)\wt F_N(\lambda)]}{\E \wt Z_N(\lambda)}.
\end{align*}
Since $\Delta_N\sim \cN(0,1)$,
\[
\E e^{\beta\Delta_N}=e^{\beta^2/2},
\qquad
\E[\Delta_N e^{\beta\Delta_N}]=\beta e^{\beta^2/2},
\]
so
\[
\beta\,\frac{\E[\Delta_N e^{\beta\Delta_N}]}{\E e^{\beta\Delta_N}}=\beta \frac{\beta e^{\beta^2/2}}{e^{\beta^2/2}}=\beta^2.
\]
Subtracting
\[
\log \E Z_N(\beta)=\log \E e^{\beta\Delta_N}+\log \E \wt Z_N(\lambda)=\frac{\beta^2}{2}+\log \E\wt Z_N(\lambda).
\]

Thus
\begin{align*}
B_N(\beta) &= \left[\beta^2 + \frac{\mathbb{E}[\widetilde{Z}_N(\lambda)\widetilde{F}_N(\lambda)]}{\mathbb{E}\widetilde{Z}_N(\lambda)}\right] - \left[\frac{\beta^2}{2} + \log\mathbb{E}\widetilde{Z}_N(\lambda)\right] \\
&= \frac{\beta^2}{2} + \left[\frac{\mathbb{E}[\widetilde{Z}_N(\lambda)\widetilde{F}_N(\lambda)]}{\mathbb{E}\widetilde{Z}_N(\lambda)} - \log\mathbb{E}\widetilde{Z}_N(\lambda)\right] \\
&= \frac{\beta^2}{2} + \widetilde{B}_N(\lambda).
\end{align*}

\end{proof}

Hence, by Lemma \ref{prop:forwardKL}, it is enough to prove
\[
\wt B_N(1)\le \frac1{12}\log N + O(1).
\]

We shall later compare $\wt B_N(1)$ with $\wt B_N(\lambda_0)$, where $\lambda_0=1-N^{-1/3}.$ The next proposition shows that $\wt B_N(\lambda_0)$ already has the the coefficient 1/12.

\begin{proposition}\label{prop:leftendpoint}
There exists a universal constant $C<\infty$ such that for every $N\ge 2$,
\[
\wt B_N(\lambda_0)\le \frac1{12}\log N + C.
\]
\end{proposition}

\begin{proof}
Fix $\lambda\in(0,1)$ and define, as before, $\wt \Lambda_{N,\lambda}(s)=\log \E e^{s\wt F_N(\lambda)},$ and
\[
\wt{\mathcal K}_{N,\lambda}(s)=
\begin{cases}
\dfrac{\wt \Lambda_{N,\lambda}(s)}{s}, & s>0,\\[1ex]
\E \wt F_N(\lambda), & s=0.
\end{cases}
\]
The same convexity argument as in Lemma \ref{lem:FNconvex} shows that $\wt{\mathcal K}_{N,\lambda}$ is convex. Hence

\begin{equation}\label{eq:wtBsecant}
\begin{aligned}
  \wt B_N(\lambda)=\wt{\mathcal K}_{N,\lambda}'(1)
&\le \wt{\mathcal K}_{N,\lambda}(2)-\wt{\mathcal K}_{N,\lambda}(1)\\
&= \frac{1}{2}\log\mathbb{E}\widetilde{Z}_N(\lambda)^2 - \log\mathbb{E}\widetilde{Z}_N(\lambda) \\
&= \frac{1}{2}\log\mathbb{E}\widetilde{Z}_N(\lambda)^2 - \frac{1}{2}\log\left(\mathbb{E}\widetilde{Z}_N(\lambda)\right)^2 \\
&= \frac{1}{2}\log\frac{\mathbb{E}[\widetilde{Z}_N(\lambda)^2]}{\left(\mathbb{E}\widetilde{Z}_N(\lambda)\right)^2}.
\end{aligned}
\end{equation}

For each fixed $\sigma\in\Sigma_N$, the random variable $\wt H_N(\sigma;J)$ is centered Gaussian with variance
\[
\E \wt H_N(\sigma;J)^2 = \frac1N\sum_{1\le i<j\le N}1 = \frac{N-1}{2}.
\]
Hence
$$
\mathbb{E}e^{\lambda\widetilde{H}_N(\sigma;J)} = \exp\left\{ \frac{\lambda^2}{2}\operatorname{Var}(\widetilde{H}_N(\sigma;J)) \right\} = \exp\left\{ \frac{\lambda^2(N-1)}{4} \right\}.$$

Summing over all $2^N$ values of $\sigma$, we obtain
\begin{equation}\label{eq:wtEZ}
\E \wt Z_N(\lambda)=2^N e^{\lambda^2(N-1)/4}.
\end{equation}
Next, for fixed $\sigma,\rho\in\Sigma_N$,
\begin{align*}
\E \wt H_N(\sigma;J)\wt H_N(\rho;J)
&=\frac1N\sum_{1\le i<j\le N} \sigma_i\sigma_j\rho_i\rho_j.
\end{align*}
Set $\tau_i=\sigma_i\rho_i\in\{-1,+1\}$. Then
\[
\sum_{1\le i<j\le N} \tau_i\tau_j
=\frac12\left(\Bigl(\sum_{i=1}^N \tau_i\Bigr)^2 - \sum_{i=1}^N \tau_i^2\right)
=\frac12\left(\Bigl(\sum_{i=1}^N \tau_i\Bigr)^2 - N\right),
\]
so
\begin{equation}\label{eq:cov-wtH}
\E \wt H_N(\sigma;J)\wt H_N(\rho;J)=
\frac{1}{N} \cdot \frac{1}{2} \left[ N^2 R(\sigma, \rho)^2 - N \right]
=\frac12\bigl(NR(\sigma,\rho)^2-1\bigr),
\end{equation}
where $R(\sigma, \rho)$ is the overlap defined above, and $\left( \sum_{i=1}^N \tau_i \right)^2 = N^2 R(\sigma, \rho)^2.$ Therefore,
\begin{align*}
\E\wt Z_N(\lambda)^2
&=\sum_{\sigma,\rho\in\Sigma_N}
\exp\left(\frac{\lambda^2}{2}\Bigl(\E \wt H_N(\sigma;J)^2 + \E \wt H_N(\rho;J)^2 + 2\E\wt H_N(\sigma;J)\wt H_N(\rho;J)\Bigr)\right)\\
&=\sum_{\sigma,\rho\in\Sigma_N}
\exp\left(\frac{\lambda^2}{2}\bigl(N-1 + N R(\sigma,\rho)^2 -1\bigr)\right)\\
&= e^{\lambda^2(N-2)/2}\sum_{\sigma,\rho\in\Sigma_N} e^{\lambda^2 N R(\sigma,\rho)^2/2}.
\end{align*}
Dividing by \eqref{eq:wtEZ} gives
\begin{equation}\label{eq:moment-ratio-raw}
\frac{\E\wt Z_N(\lambda)^2}{(\E\wt Z_N(\lambda))^2}
= e^{-\lambda^2/2}\,2^{-2N}\sum_{\sigma,\rho\in\Sigma_N} e^{\lambda^2 N R(\sigma,\rho)^2/2}.
\end{equation}
Let $\sigma,\rho$ now be independent uniform random spins, and define $\xi_i=\sigma_i\rho_i$. Then $(\xi_i)_{i=1}^N$ are i.i.d. Rademacher random variables, and let $S_N=\sum_{i=1}^N \xi_i,$ then $R(\sigma,\rho)=S_N/N$. Thus \eqref{eq:moment-ratio-raw} becomes
\begin{equation}\label{eq:moment-ratio}
\frac{\E\wt Z_N(\lambda)^2}{(\E\wt Z_N(\lambda))^2}
= e^{-\lambda^2/2}\, \E\exp\!\left(\frac{\lambda^2}{2N}S_N^2\right).
\end{equation}
To bound the remaining expectation, let $G\sim \cN(0,1)$. The Hubbard--Stratonovich identity gives
\begin{equation}
    \exp\!\left(\frac{x^2}{2}\right)
    =
    \E_G e^{Gx}.
    \label{eq:gaussian-mgf}
\end{equation}
Applying \eqref{eq:gaussian-mgf} with $x=\lambda S_N/\sqrt N$ and using independence of the $\xi_i$,
\begin{align*}
\E\exp\!\left(\frac{\lambda^2}{2N}S_N^2\right)
&=\E_G\E\exp\!\left(\frac{\lambda G}{\sqrt N}S_N\right)\\
&=\E_G\prod_{i=1}^N \E e^{\lambda G\xi_i/\sqrt N}
=\E_G\Bigl[\cosh\!\left(\frac{\lambda G}{\sqrt N}\right)^N\Bigr].
\end{align*}
Since $\cosh x\le e^{x^2/2}$ for every $x\in\R$,
\begin{equation}
\begin{aligned}
    \E \exp\!\left(\frac{\lambda^2}{2N}S_N^2\right)
    \leq
    \E_G e^{\lambda^2 G^2/2}=
    (1-\lambda^2)^{-1/2},
    \qquad 0<\lambda<1 .
\end{aligned}
\label{eq:SN-exp-square-bound}
\end{equation}
Combining \eqref{eq:SN-exp-square-bound} with \eqref{eq:wtBsecant} and \eqref{eq:moment-ratio}, we obtain
\[
\wt B_N(\lambda)
\le \frac12\log\left(e^{-\lambda^2/2}(1-\lambda^2)^{-1/2}\right)
= -\frac{\lambda^2}{4} + \frac14\log\frac1{1-\lambda^2}
\le \frac14\log\frac1{1-\lambda^2}.
\]
Now substitute $\lambda=\lambda_0=1-N^{-1/3}$. Then
\[
1-\lambda_0^2 = 1-(1-N^{-1/3})^2 = 2N^{-1/3}-N^{-2/3}.
\]
Hence, for $N\ge 2$, $1-\lambda_0^2 \ge N^{-1/3},$ $1-\lambda_0^2 \le 2N^{-1/3},$ and therefore
\[
\wt B_N(\lambda_0)
\le \frac14\log\frac1{1-\lambda_0^2}
\le \frac14\left(\frac13\log N + \log 2\right).
\]

\end{proof}
We next relate \(\widetilde B_N\) to a Gaussian observation model in information-theoretic regime. Fix $\lambda\ge 0$. Let $E_N=\bigl\{\{i,j\}: 1\le i<j\le N\bigr\},$ and $M=|E_N|=\binom N2.$ Let $\nu_0$ be the law of an $M$-dimensional standard Gaussian vector $Y=(Y_{ij})_{1\le i<j\le N}$. Independently, let $\theta=(\theta_1,\dots,\theta_N)$ be uniform on $\Sigma_N$, and given $\theta$ define the observation model
\begin{equation}\label{eq:sync-model}
Y_{ij}=\frac{\lambda}{\sqrt N}\theta_i\theta_j + G_{ij},
\qquad 1\le i<j\le N,
\end{equation}
where $(G_{ij})_{i<j}$ are i.i.d. standard Gaussians, independent of $\theta$. Let $\mathbb P_\lambda$ denote the joint law of $(\theta,Y)$, let $\nu_{\lambda,x}$ denote the conditional law of $Y$ given $\theta=x$, and let $\nu_\lambda$ denote the marginal law of $Y$.

\begin{lemma}\label{prop:KLrepresentation}
For every $\lambda\ge 0$,
\[
\wt B_N(\lambda)=\KL(\nu_\lambda\|\nu_0).
\]
\end{lemma}

\begin{proof}
Fix $x\in\Sigma_N$. Under $\nu_{\lambda,x}$, $Y_{ij} \sim \mathcal{N}\left(\frac{\lambda}{\sqrt{N}}x_i x_j, 1\right).$ Under $\nu_0$, $Y_{ij} \sim \mathcal{N}(0, 1).$ The ratio of one-dimensional Gaussian densities is
\[
\frac{\exp\left\{-\frac{1}{2}(y-\mu)^2\right\}}{\exp\left\{-\frac{1}{2}y^2\right\}} = \exp\left\{\frac{\lambda}{\sqrt{N}}x_i x_j y - \frac{\mu^2}{2}\right\}.
\]

Therefore, the density ratio on a single edge is $\exp\left\{\frac{\lambda}{\sqrt{N}}y_{ij}x_i x_j - \frac{\lambda^2}{2N}\right\},$ under the conditional law $\nu_{\lambda,x}$, the density with respect to $\nu_0$ is
\begin{align*}
\frac{d\nu_{\lambda,x}}{d\nu_0}(y)
&=\prod_{1\le i<j\le N}
\exp\left(\frac{\lambda}{\sqrt N}y_{ij}x_ix_j - \frac{\lambda^2}{2N}\right)\\
&=\exp\left(\lambda\,\wt H_N(x;y)-\frac{\lambda^2 M}{2N}\right).
\end{align*}
Averaging over the uniform prior on $\theta$ gives
\begin{align*}
\frac{d\nu_\lambda}{d\nu_0}(y)&=2^{-N} \sum_{x \in \Sigma_N} \frac{d\nu_{\lambda,x}}{d\nu_0}(y)\\
&=2^{-N}\sum_{x\in\Sigma_N}
\exp\left(\lambda\,\wt H_N(x;y)-\frac{\lambda^2 M}{2N}\right)\\
&=\frac{\wt Z_N(\lambda;y)}{2^N e^{\lambda^2 M/(2N)}}.
\end{align*}
By \eqref{eq:wtEZ},
\[
\E_{\nu_0} \wt Z_N(\lambda)= \exp \left\{ \frac{\lambda^2}{2} \cdot \frac{M}{N} \right\}=2^N e^{\lambda^2 M/(2N)},
\]
so
\begin{equation}\label{eq:RNdensity}
\frac{d\nu_\lambda}{d\nu_0}(y)=\frac{\wt Z_N(\lambda;y)}{\E_{\nu_0}\wt Z_N(\lambda)}.
\end{equation}
Hence
\begin{align*}
\KL(\nu_\lambda\|\nu_0)
&=\E_{\nu_\lambda}\log\frac{d\nu_\lambda}{d\nu_0}(Y)
=\frac{\E_{\nu_0}\bigl[\wt Z_N(\lambda)\log(d\nu_\lambda/d\nu_0)\bigr]}{\E_{\nu_0}\wt Z_N(\lambda)}\\
&=\frac{\E_{\nu_0}\bigl[\wt Z_N(\lambda)(\wt F_N(\lambda)-\log \E_{\nu_0}\wt Z_N(\lambda))\bigr]}{\E_{\nu_0}\wt Z_N(\lambda)}\\
&=\frac{\E[\wt Z_N(\lambda)\wt F_N(\lambda)]}{\E \wt Z_N(\lambda)} - \log \E\wt Z_N(\lambda)
=\wt B_N(\lambda).
\end{align*}
\end{proof}

Define the edge-spin vector $U=(U_{ij})_{1\le i<j\le N}$ by $U_{ij}=\theta_i\theta_j.$ Then \eqref{eq:sync-model} takes the form $Y=\sqrt\gamma\,U + G,$ where $\gamma=\frac{\lambda^2}{N},$ and $G\sim \cN(0,I_M)$ is independent of $U$. For $1\le i<j\le N$, define the posterior correlation
\begin{equation}\label{eq:mijdef}
m_{ij}^{(\lambda)}(Y)=\E_{\mathbb P_\lambda}[\theta_i\theta_j\mid Y].
\end{equation}
We can show that $\wt B_N'(\lambda)$ is expressed in terms of $m_{ij}^{(\lambda)}(Y)$.

\begin{proposition}\label{prop:derivative}
For every $\lambda\ge 0$,
\[
\frac{d}{d\lambda}\wt B_N(\lambda)=\frac{\lambda}{N}\sum_{1\le i<j\le N} \E_{\mathbb P_\lambda}\bigl[m_{ij}^{(\lambda)}(Y)^2\bigr].
\]
\end{proposition}

\begin{proof}
All expectations and conditional expectations in this proof are taken with respect to the joint law $\mathbb P_\lambda$.
Set $I_N(\gamma)=I(U;Y),$ $\KL(\nu_\lambda\|\nu_0)=\wt B_N(\lambda).$  We first relate $\mathcal D_N$ and $I_N$. Conditional on \(U=u\), the observation \(Y\) has distribution \(\mathcal N(\sqrt{\gamma},u,I_M)\). Hence, \
\[
\frac{dP(Y\mid U)}{d\nu_0}(Y)=\exp\left(\sqrt\gamma\,U\cdot Y - \frac\gamma2\|U\|^2\right)
\]
and $\|U\|^2=M$ deterministically, we have

\begin{equation}
\begin{split}
I_N(\gamma)+\KL(\nu_\lambda\|\nu_0)
&=\E\log\frac{dP(Y\mid U)}{d\nu_0}(Y)\\
&=\E\left[\sqrt\gamma\,U\cdot Y - \frac\gamma2 M\right].
\end{split}
\label{eq:ID_decomposition}
\end{equation}


%

But under $Y=\sqrt\gamma\,U+G$, $$U \cdot Y = U \cdot (\sqrt{\gamma}U + G) = \sqrt{\gamma}\|U\|^2 + U \cdot G.$$

Thus, 
\[
\E[U\cdot Y]= \sqrt{\gamma} \, \mathbb{E}\|U\|^2 + \mathbb{E}[U \cdot G]=\sqrt\gamma\,\E\|U\|^2 = \sqrt\gamma\,M,
\]
so
\begin{equation}\label{eq:DplusI}
\KL(\nu_\lambda\|\nu_0)=\frac{\gamma M}{2} - I_N(\gamma).
\end{equation}
Now apply Theorem \ref{thm:IMSE}. Since $U$ is an $M$-dimensional random vector with finite second moment,
\[
\frac{d}{d\gamma}I_N(\gamma)
=\frac12\sum_{1\le i<j\le N}\E\Bigl[(U_{ij}-\E[U_{ij}\mid Y])^2\Bigr].
\]
Because $U_{ij}^2=1$ and $\E[U_{ij}\mid Y]=m_{ij}^{(\lambda)}(Y)$,
\begin{align*}
\E\Bigl[(U_{ij}-\E[U_{ij}\mid Y])^2\Bigr]
&=\E\bigl[U_{ij}^2\bigr]-2\E\bigl[U_{ij}m_{ij}^{(\lambda)}(Y)\bigr]+\E\bigl[m_{ij}^{(\lambda)}(Y)^2\bigr]\\
&=1-\E\bigl[m_{ij}^{(\lambda)}(Y)^2\bigr],
\end{align*}
and
\[
\E\bigl[U_{ij}m_{ij}^{(\lambda)}(Y)\bigr]
=\E\Bigl[\E[U_{ij}\mid Y]m_{ij}^{(\lambda)}(Y)\Bigr]
=\E\bigl[m_{ij}^{(\lambda)}(Y)^2\bigr].
\]
Therefore,
\[
\frac{d}{d\gamma}I_N(\gamma)
=\frac12\sum_{1\le i<j\le N}\left(1-\E\bigl[m_{ij}^{(\lambda)}(Y)^2\bigr]\right).
\]
Differentiating \eqref{eq:DplusI} yields
\[
\frac{d}{d\gamma}\KL(\nu_\lambda\|\nu_0)
=\frac M2-\frac{d}{d\gamma}I_N(\gamma)= \frac{M}{2} - \frac{1}{2} \sum_{1 \le i < j \le N} \left( 1 - \mathbb{E}[m_{ij}^{(\lambda)}(Y)^2] \right)
=\frac12\sum_{1\le i<j\le N}\E\bigl[m_{ij}^{(\lambda)}(Y)^2\bigr].
\]
Finally,
\[
\frac{d\gamma}{d\lambda}=\frac{2\lambda}{N},
\]
so by the chain rule,
\[
\frac{d}{d\lambda}\wt B_N(\lambda)
=\frac{2\lambda}{N}\cdot \frac12\sum_{1\le i<j\le N}\E\bigl[m_{ij}^{(\lambda)}(Y)^2\bigr].
\]
\end{proof}

We now connect the posterior correlation to bond percolation.

\begin{lemma}\label{lem:I2equalsm2}
Let $(A,B,Y)$ be random variables with $A,B\in\{-1,+1\}$ such that, almost surely,
\[
\Pbb(A=1\mid Y)=\Pbb(A=-1\mid Y)=\frac12,
\]
and
\[
\Pbb(B=1\mid Y)=\Pbb(B=-1\mid Y)=\frac12.
\]
Then
\[
\Ichi(A;B\mid Y)=\E\bigl[\E[AB\mid Y]^2\bigr].
\]
\end{lemma}

\begin{proof}
Set $m(Y)=\E[AB\mid Y].$ Since the conditional marginals of $A$ and $B$ are uniform, the conditional law of $(A,B)$ given $Y$ is necessarily
\[
\Pbb(A=a,B=b\mid Y)=\frac14\bigl(1+ab\,m(Y)\bigr),
\qquad a,b\in\{-1,+1\}.
\]
Therefore,
\begin{align*}
\chi^2\bigl(P_{A,B\mid Y},P_{A\mid Y}P_{B\mid Y}\bigr)
&=\sum_{a,b\in\{-1,+1\}} \frac{\left(\frac14(1+ab\,m(Y)) - \frac14\right)^2}{\frac14}\\
&=\sum_{a,b\in\{-1,+1\}} \frac{(ab\,m(Y))^2}{16}\cdot 4\\
&=m(Y)^2.
\end{align*}
Taking expectations proves the claim.
\end{proof}

\begin{corollary}\label{cor:AB-specialized}
For every $\lambda\ge 0$ and every $1\le u<v\le N$,
\begin{equation}
\label{eq:eEbound}
\E_{\mathbb P_\lambda}\bigl[m_{uv}^{(\lambda)}(Y)^2\bigr]
\le \Pbb\bigl(u\leftrightarrow v \text{ in } \mathcal{G}_{N,p_N(\lambda)}\bigr),
\end{equation}

where $p_N(\lambda)=\Ichi(\theta_1;\theta_2\mid Y_{12})$ in the single edge observation model associated with \eqref{eq:sync-model}. In particular, the right-hand side of \eqref{eq:eEbound} is the connection probability in Erd\H{o}s-R\'enyi graph $\mathcal{G}_{N,p_N(\lambda)}$ on $N$ vertices with edge density $p_N(\lambda)$.
\end{corollary}

\begin{proof}
We consider the complete graph on $\{1,\dots,N\}$ endowed with the symmetric Gaussian edge channel $Y_e = \sqrt\gamma\,s + G_e,$ $s\in\{-1,+1\},\ G_e\sim \cN(0,1).$ Thus, the observation channel for a single edge is

$$Q(\mathrm{d}y \mid s) = \frac{1}{\sqrt{2\pi}} \exp \left\{ - \frac{(y - \sqrt{\gamma}s)^2}{2} \right\} \mathrm{d}y.$$

Moreover, take $T(y) = -y.$ Then $Q(A \mid +1) = \mathbb{P}(\sqrt{\gamma} + G \in A),$ while \[
Q(T(A) \mid -1) = \mathbb{P}(-\sqrt{\gamma} + G \in -A) = \mathbb{P}(\sqrt{\gamma} - G \in A).
\]
Since $G \stackrel{d}{=} -G$, we have
\[
\mathbb{P}(\sqrt{\gamma} - G \in A) = \mathbb{P}(\sqrt{\gamma} + G \in A).
\]

Hence the channel is symmetric, and Theorem \ref{thm:AB} applies.  Since the likelihood depends only on the products $\theta_i \theta_j$, $\theta = (\theta_1, \dots, \theta_N)$ and $-\theta = (-\theta_1, \dots, -\theta_N)$ yield the same edge products $(-\theta_i)(-\theta_j) = \theta_i \theta_j.$ Therefore,
\[
\mathbb{P}_\lambda(\theta = a \mid Y) = \mathbb{P}_\lambda(\theta = -a \mid Y)
\]
holds almost surely. Thus, for any vertex $u$,
\[
\mathbb{P}_\lambda(\theta_u = 1 \mid Y) = \mathbb{P}_\lambda(\theta_u = -1 \mid Y) = \frac{1}{2}.
\]


Then Lemma \ref{lem:I2equalsm2} yields
\[
\Ichi(\theta_u;\theta_v\mid Y)=\mathbb{E}_{\mathbb{P}_\lambda} \left[ \left( \mathbb{E}_{\mathbb{P}_\lambda} [\theta_u \theta_v \mid Y] \right)^2 \right]=\E_{\mathbb P_\lambda}\bigl[m_{uv}^{(\lambda)}(Y)^2\bigr].
\]
Similarly, for a single edge $e=\{i,j\}$, the individual conditional marginals of $\theta_i$ and $\theta_j$ given $Y_e$ are also uniform, so
\[
p_e=\Ichi(\theta_i;\theta_j\mid Y_e)=\mathbb{E} \left[ \left( \mathbb{E} [ \theta_i \theta_j \mid Y_{ij} ] \right)^2 \right].
\]
is the same for every edge and equals $p_N(\lambda)$. Thus the bond percolation model in Theorem \ref{thm:AB} is exactly the Erd\H{o}s-R\'enyi graph $\mathcal{G}_{N,p_N(\lambda)}$,and then we have 

$$\mathbb{E}_{\mathbb{P}_\lambda} \left[ |m_{uv}^{(\lambda)}(Y)|^2 \right] = I_2(\theta_u; \theta_v \mid Y) \le \mathbb{P}(u \leftrightarrow v \text{ in } \mathcal{G}_{N, p_N(\lambda)}).$$
\end{proof}

The remaining is to analyze the single edge open probability.

\begin{lemma}\label{lem:singleedge-expansion}
There exists a universal constant $C<\infty$ such that the following holds for all $N\ge 1$ and all $\lambda\in[0,1]$.
\begin{equation}\label{eq:pN-expansion}
p_N(\lambda)=\gamma + \varepsilon_N(\lambda),
\end{equation}
where $|\varepsilon_N(\lambda)|\le C\gamma^2.$ In particular, for $\lambda\in[\lambda_0,1]=[1-N^{-1/3},1],$ there exists a universal constant $A<\infty$ such that
\begin{equation}\label{eq:pN-critical-window}
p_N(\lambda)\le \frac1N + A N^{-4/3}.
\end{equation}
\end{lemma}

\begin{proof}
Consider a single edge and let $
S=\theta_1\theta_2\in\{-1,+1\}
$ (which is uniform), and
\[
Y_e=\sqrt\gamma\,S+G_e,
\qquad G_e\sim \cN(0,1).
\]
A direct Bayes calculation gives
\[
\E[S\mid Y=y]
=\frac{e^{-(y-\sqrt\gamma)^2/2}-e^{-(y+\sqrt\gamma)^2/2}}
{e^{-(y-\sqrt\gamma)^2/2}+e^{-(y+\sqrt\gamma)^2/2}}
=\tanh(\sqrt\gamma\,y).
\]
By Lemma \ref{lem:I2equalsm2} applied to $(\theta_1,\theta_2,Y)$,
\[
p_N(\lambda)=\Ichi(\theta_1;\theta_2\mid Y)= \mathbb{E} \left[ \left( \mathbb{E}[\theta_1 \theta_2 \mid Y] \right)^2 \right]= \mathbb{E} \left[ \left( \mathbb{E}[S \mid Y] \right)^2 \right]=\E\bigl[\tanh^2(\sqrt\gamma\,Y)\bigr].
\]
Since $S$ is symmetric and $\tanh^2$ is even,
\[
p_N(\lambda)=\mathbb{E} \left[ \tanh^2(\gamma S + \sqrt{\gamma} G) \right]=\E\bigl[\tanh^2(\gamma + \sqrt\gamma G)\bigr].
\]
We claim that there exists a universal constant $C_0$ such that for every $x\in\R$,
\begin{equation}\label{eq:tanh-global}
\bigl|\tanh^2 x - x^2\bigr|\le C_0 x^4.
\end{equation}
Indeed, for $|x|\le 1$, \eqref{eq:tanh-global} follows from the Taylor expansion
\[
\tanh^2 x = x^2 - \frac23 x^4 + O(x^6),
\]
while for $|x|\ge 1$ we simply use
\[
\bigl|\tanh^2 x - x^2\bigr|\le 1+x^2\le 2x^4.
\]
Applying \eqref{eq:tanh-global} with $x=\gamma+\sqrt\gamma G$, we get
\[
\bigl|p_N(\lambda)-\E(\gamma+\sqrt\gamma G)^2\bigr|
\le C_0\,\E(\gamma+\sqrt\gamma G)^4.
\]
Now
\[
\E(\gamma+\sqrt\gamma G)^2 = \gamma^2 + \gamma,
\]
and for $0\le \gamma\le 1$, a direct expansion using $\E G=0$, $\E G^2=1$, $\E G^4=3$ yields
\[
\E(\gamma+\sqrt\gamma G)^4=\gamma^4 + 4\gamma^3 \mathbb{E}G + 6\gamma^3 \mathbb{E}G^2 + 4\gamma^{5/2} \mathbb{E}G^3 + \gamma^2 \mathbb{E}G^4 = \gamma^4 + 6\gamma^3 + 3\gamma^2 \le 10\gamma^2.
\]

\begin{align*}
|p_N(\lambda) - (\gamma + \gamma^2)| &= |\mathbb{E}[\tanh^2(\gamma + \sqrt{\gamma} G)] - \mathbb{E}[(\gamma + \sqrt{\gamma} G)^2]| \\
&= |\mathbb{E}[\tanh^2(\gamma + \sqrt{\gamma} G) - (\gamma + \sqrt{\gamma} G)^2]| \\
&\le \mathbb{E}[|\tanh^2(\gamma + \sqrt{\gamma} G) - (\gamma + \sqrt{\gamma} G)^2|] \\
&\le C_0 \mathbb{E}[(\gamma + \sqrt{\gamma} G)^4]\le 10 C_0 \gamma^2.
\end{align*}

So,
\[
|p_N(\lambda) - \gamma| \le |p_N(\lambda) - (\gamma + \gamma^2)| + \gamma^2 \le (10 C_0 + 1)\gamma^2.
\]

Hence
\[
p_N(\lambda)=\gamma + O(\gamma^2),
\]
with an absolute implied constant, proving \eqref{eq:pN-expansion}.

For the second claim, if $\lambda\in[1-N^{-1/3},1]$, then $\gamma=\lambda^2/N\le 1/N$, so by \eqref{eq:pN-expansion},
\[
p_N(\lambda)\le \frac1N + \frac{C}{N^2}\le \frac1N + C N^{-4/3}
\]

Thus, \eqref{eq:pN-critical-window} holds with $A=C$.
\end{proof}

Then we obtain the critical window upper bound.

\begin{lemma}\label{prop:derivative-bound}
There exists a universal constant $C<\infty$ such that for every $N\ge 1$ and every $\lambda\in[\lambda_0,1]$,
\[
\frac{d}{d\lambda}\wt B_N(\lambda)\le C N^{1/3}.
\]
Consequently,
\[
\wt B_N(1)\le \wt B_N(\lambda_0)+C.
\]
\end{lemma}




\begin{proof}
Fix $\lambda\in[\lambda_0,1]$. By Proposition \ref{prop:derivative} and symmetry of the complete graph,
\begin{equation}
\frac{d}{d\lambda}\wt B_N(\lambda)
=\frac{\lambda}{N}\binom N2\,\E_{\mathbb P_\lambda}\bigl[m_{12}^{(\lambda)}(Y)^2\bigr].
\label{eq:derivative}
\end{equation}

By Corollary \ref{cor:AB-specialized},
\[
\E_{\mathbb P_\lambda}\bigl[m_{12}^{(\lambda)}(Y)^2\bigr]
\le \Pbb\bigl(1\leftrightarrow 2 \text{ in } \mathcal{G}_{N,p_N(\lambda)}\bigr).
\]
By Lemma \ref{lem:singleedge-expansion}, 
\[
p_N(\lambda)\le \frac1N + A N^{-4/3}
\]
for some universal $A$. Since connectivity is monotone non-decreasing with respect to edge addition, by \cite[Corollary~5.3]{JansonSpencer2007} with \(q=2\), applied at the critical window with
parameter \(A\), 

$$\mathbb{E} \sum_{\mathcal{C} \in \mathrm{Comp}(\mathcal{G}_{N, p_N(\lambda)})} |\mathcal{C}|^2 \le \mathbb{E} \sum_{\mathcal{C} \in \mathrm{Comp}(\mathcal{G}_{N, \frac1N + A N^{-4/3}})} |\mathcal{C}|^2\le C_A N^{4/3} .$$
where \(\mathrm{Comp}(\mathcal G)\) denotes the set of connected components
of \(\mathcal G\).
Let $\mathcal C_{\mathcal{G}}(1)$ denotes the connected component of vertex $1$, i.e.
$$\sum_{v=1}^N |\mathcal{C}_{\mathcal{G}}(v)| = \sum_{\mathcal{C} \in \mathrm{Comp}(\mathcal{G})} |\mathcal{C}|^2.$$

Since the vertices are exchangeable, $\E \sum_{v=1}^{N} |\mathcal{C}_{\mathcal{G}}(v)| = N \E |\mathcal{C}_{\mathcal{G}}(1)|.$ Thus
\[
\E |\mathcal C_{\mathcal{G}}(1)| = \frac1N \E\sum_{\mathcal C} |\mathcal C|^2 \le C_A N^{1/3}.
\]

Moreover, observe that
\[
|\mathcal C_{\mathcal{G}}(1)| = 1 + \sum_{v=2}^N \1_{\{1\leftrightarrow v\}},
\]
so by symmetry,
\[
\E |\mathcal C_{\mathcal{G}}(1)| = 1 + (N-1)\Pbb(1\leftrightarrow 2 \text{ in } \mathcal{G}).
\]
Hence, for two distinct fixed vertices $1,2$,
\[
\Pbb(1\leftrightarrow 2 \text{ in } \mathcal{G}_{N,p_N(\lambda)})=\frac{\E |\mathcal C_{\mathcal{G}}(1)|-1}{N-1}\le C_A N^{-2/3}.
\]


where $C_A$ depends only on the fixed constant $A$ and is independent of $N$ and $\lambda$. Substituting back into \eqref{eq:derivative},

\[
\frac{d}{d\lambda}\wt B_N(\lambda)
\le \frac{1}{N}\binom N2\, C_A N^{-2/3}
\le C N^{1/3}
\]
for a universal $C$.
Integrating over $[\lambda_0,1]$ and using $1-\lambda_0=N^{-1/3}$ gives

\[
\begin{aligned}
\wt B_N(1)-\wt B_N(\lambda_0)
&= \int_{\lambda_0}^1 \frac{d}{d\lambda} \widetilde{B}_N(\lambda) \, d\lambda \le \int_{\lambda_0}^1 C N^{1/3}\,d\lambda \\
&= C N^{1/3} (1 - \lambda_0) = C N^{1/3} N^{-1/3} = C.
\end{aligned}
\]
\end{proof}

\begin{proof}[Proof of upper bound]
For $N=1$ the claim is immediate after enlarging the final constant, so we may assume $N\ge 2$.
By Lemma \ref{prop:forwardKL} and \ref{prop:Bdecomp},
\[
\Var\bigl(F_N(\beta_c)\bigr)
\le 2B_N(\beta_c)
= 2\left(\frac14+\wt B_N(1)\right).
\]
By Lemma \ref{prop:derivative-bound} and Proposition \ref{prop:leftendpoint},
\[
\wt B_N(1)
\le \wt B_N(\lambda_0)+C_1
\le \frac1{12}\log N + C_2.
\]
Therefore,
\[
\Var\bigl(F_N(\beta_c)\bigr)
\le 2\left(\frac14 + \frac1{12}\log N + C_2\right)
= \frac16\log N + \left(\frac12+2C_2\right).
\]
\end{proof}

\subsection{Proof of lower bound}\label{subsec:critical-lower-bound}

\medskip
\noindent

The argument of lower bound combines the Gaussian convexity with the entropy estimate above, and deal with the remaining negative moment to a one replica spherical integral and a dimension shift identity for the GOE eigenvalue density.

\begin{theorem}\label{thm:critical-lower-bound}
There exists a universal constant $C<\infty$ such that for every integer $N\ge 2$,
\[
\Var\bigl(F_N(\beta_c)\bigr)\ge \frac16\log N-C.
\]
\end{theorem}

\begin{remark}\label{rem:du-huang}
Du and Huang \cite{DuHuang2026} prove the same critical variance asymptotic, together with a central limit theorem and sharp estimates for the critical overlap. Our lower bound argument also uses the spherical SK partition function and GOE soft edge asymptotics, but the strategy is different. We apply the Gaussian convexity principle for arbitrary real parameters, thereby reducing the proof of the variance lower bound to a lower bound on the inverse second moment \(\mathbb E\mathcal Z_N^{-2}\). We then compare this inverse moment with its spherical counterpart via Haar averaging and extract the required polynomial factor from a one replica contour formula combined with a dimension shift identity for GOE eigenvalue densities. 

\end{remark}

We first rewrite the critical model in GOE notation. For $1\le i<j\le N$, set $W_{ij}=W_{ji}=\frac{g_{ij}+g_{ji}}{\sqrt{2N}},$ and $W_{ii}=\sqrt{\frac2N}\,g_{ii}.$ Then $W$ is a $N\times N$ GOE matrix. Its upper triangular entries are independent with $W_{ij}\sim\mathcal N(0,N^{-1})$ for $i<j,$ and $W_{ii}\sim\mathcal N(0,2N^{-1}).$
 
 Moreover, for every $\sigma\in\Sigma_N$,
\begin{align*}
\frac12\sigma^{\mathsf T}W\sigma
&=\sum_{1\le i<j\le N}W_{ij}\sigma_i\sigma_j+\frac12\sum_{i=1}^N W_{ii}\\
&=\frac1{\sqrt{2N}}\sum_{1\le i<j\le N}(g_{ij}+g_{ji})\sigma_i\sigma_j
  +\frac1{\sqrt{2N}}\sum_{i=1}^N g_{ii}\\
&=\beta_c H_N(\sigma;g).
\end{align*}
Define the normalized partition function and its logarithm by
\[
\mathcal Z_N(W)=2^{-N}\sum_{\sigma\in\Sigma_N}
\exp\left\{\frac12\sigma^{\mathsf T}W\sigma\right\},
\]
and $$\mathcal F_N(W)=\log\mathcal Z_N(W).$$
Thus
\begin{equation}\label{eq:lower-full-to-goe}
F_N(\beta_c)=N\log 2+\mathcal F_N,
\qquad
\Var\bigl(F_N(\beta_c)\bigr)=\Var(\mathcal F_N).
\end{equation}
For a fixed $\sigma$, the Gaussian variable $\frac12\sigma^{\mathsf T}W\sigma$ has variance
\[
\sum_{1\le i<j\le N}\frac1N+\frac14\sum_{i=1}^N\frac2N
=\frac{N-1}{2}+\frac12=\frac N2.
\]
It follows that
\begin{equation}\label{eq:lower-first-moment}
\E\mathcal Z_N=e^{N/4}.
\end{equation}

By Lemma~\ref{prop:Bdecomp}, Proposition~\ref{prop:leftendpoint}, and Lemma~\ref{prop:derivative-bound}, there is a universal constant $C$ such that
\begin{equation}\label{eq:lower-entropy-input}
B_N(\beta_c)=\frac14+\wt B_N(1)\le \frac1{12}\log N+C.
\end{equation}

For $s\in\mathbb R$, now set
\[
\widehat{\mathcal K}_N(s)=
\begin{cases}
\dfrac1s\log\E e^{s\mathcal F_N},&s\ne0,\\[1ex]
\E\mathcal F_N,&s=0.
\end{cases}
\]

Then we have the following result using \cite[Theorem~1]{Chen2024GaussianConvexity} at the negative replica value $s=-2$.

\begin{lemma}\label{lem:lower-three-point-convexity}
The function $\widehat{\mathcal K}_N$ is finite and convex on $\mathbb R$, and
\begin{equation}\label{eq:lower-three-point-convexity}
\Var(\mathcal F_N)
\ge
\widehat{\mathcal K}_N(1)-\widehat{\mathcal K}_N(-2)-\widehat{\mathcal K}_N'(1).
\end{equation}
Furthermore,
\begin{equation}\label{eq:lower-convexity-reduction}
\Var(\mathcal F_N)
\ge
\frac N4+\frac12\log\E\mathcal Z_N^{-2}-B_N(\beta_c).
\end{equation}
\end{lemma}

\begin{proof}
View $W$ as a fixed linear image of the $N(N+1)/2$ independent standard Gaussian coordinates from which its entries are constructed. Since $\mathcal F_N$ is the logarithm of a finite sum of exponentials of affine functions of these coordinates, it is convex.

We next verify the exponential integrability at every real value of $s$. The upper bound
\[
\mathcal F_N(W)
\le
\sum_{1\le i<j\le N}|W_{ij}|+\frac12\sum_{i=1}^N|W_{ii}|
\]
shows that $\E e^{s\mathcal F_N}<\infty$ for $s>0$. For the negative half-line, Jensen's inequality with respect to the uniform measure on $\Sigma_N$ gives
\begin{equation}
\begin{split}
\mathcal F_N(W)
&=\log\left(2^{-N}\sum_{\sigma\in\Sigma_N}
 e^{\sigma^{\mathsf T}W\sigma/2}\right)\\
&\ge 2^{-N}\sum_{\sigma\in\Sigma_N}\frac12\sigma^{\mathsf T}W\sigma
=\frac12\operatorname{Tr}W.
\end{split}
\label{eq:lower_bound}
\end{equation}
The last equality of \eqref{eq:lower_bound} follows because $2^{-N}\sum_\sigma\sigma_i\sigma_j=0$ for $i\ne j$ and equals one for $i=j$. Hence, if $s<0$,
\[
e^{s\mathcal F_N}\le e^{s\operatorname{Tr}W/2},
\]

note that $\operatorname{Tr}W$ is Gaussian, then $e^{s\operatorname{Tr}W/2}$ is integrable. Therefore $\widehat{\mathcal K}_N$ is finite on $\mathbb R$, and the Theorem \ref{thm:Chen} implies its convexity on $\mathbb R$.

Since $\mathcal F_N$ has exponential moments in a neighborhood of the origin, its logarithmic moment generating function is analytic there. Expanding at zero gives
\[
\log\E e^{s\mathcal F_N}
=s\E\mathcal F_N+\frac{s^2}{2}\Var(\mathcal F_N)+O(s^3),
\]
and consequently
\begin{equation}\label{eq:lower-Kprime-zero}
\widehat{\mathcal K}_N'(0)=\frac12\Var(\mathcal F_N).
\end{equation}
For a differentiable convex function, every secant slope ending at a point is bounded above by the derivative at that point. Applying this on $[-2,0]$ and $[0,1]$ yields
\begin{equation}
\widehat{\mathcal K}_N(0)-\widehat{\mathcal K}_N(-2)\le 2\widehat{\mathcal K}_N'(0),
\label{eq:K_bound_1}
\end{equation}
and
\begin{equation}
\widehat{\mathcal K}_N(1)-\widehat{\mathcal K}_N(0)\le \widehat{\mathcal K}_N'(1).
\label{eq:K_bound_2}
\end{equation}
Adding \eqref{eq:K_bound_1}, \eqref{eq:K_bound_2} and using \eqref{eq:lower-Kprime-zero} gives \eqref{eq:lower-three-point-convexity}.

We now compute the three terms in \eqref{eq:lower-three-point-convexity}. Firstly, equation~\eqref{eq:lower-first-moment} gives
\[
\widehat{\mathcal K}_N(1)=\log\E\mathcal Z_N=\frac N4,
\]
whereas
\[
\widehat{\mathcal K}_N(-2)=-\frac12\log\E\mathcal Z_N^{-2}.
\]
Finally,
\begin{align*}
\widehat{\mathcal K}_N'(1)
&=\frac{\E[\mathcal Z_N\mathcal F_N]}{\E\mathcal Z_N}
  -\log\E\mathcal Z_N\\
&=\frac{\E[Z_N(\beta_c)F_N(\beta_c)]}{\E Z_N(\beta_c)}
  -\log\E Z_N(\beta_c)\\
&=B_N(\beta_c).
\end{align*}
In the second equality, the terms produced by $\mathcal Z_N=2^{-N}Z_N(\beta_c)$ and $\mathcal F_N=F_N(\beta_c)-N\log2$ cancel. Substitution into \eqref{eq:lower-three-point-convexity} proves \eqref{eq:lower-convexity-reduction}.
\end{proof}

It remains to obtain the polynomial factor in the negative second moment of $\mathcal Z_N$. Let $\mathbb S_N=\{x\in\mathbb R^N:\|x\|^2=N\},$ and let $\nu_N$ be normalized surface measure on $\mathbb S_N$, and define the spherical partition function
\[
\mathcal Z_N^{\mathrm{sph}}(W)
=\int_{\mathbb S_N}
\exp\left\{\frac12x^{\mathsf T}Wx\right\}\,d\nu_N(x).
\]

\begin{lemma}\label{lem:lower-haar-jensen}
For every $N\ge1$,
\begin{equation}\label{eq:lower-haar-jensen}
\E\mathcal Z_N^{-2}\ge \E\bigl(\mathcal Z_N^{\mathrm{sph}}\bigr)^{-2}.
\end{equation}
\end{lemma}

\begin{proof}
Fix $W$ and let $O$ be an independent Haar-distributed orthogonal matrix. For every fixed $\sigma\in\Sigma_N$, the vector $O\sigma$ is uniformly distributed on $\mathbb S_N$. Therefore,
\begin{align*}
\E_O\mathcal Z_N(O^{\mathsf T}WO)
&=2^{-N}\sum_{\sigma\in\Sigma_N}
\E_O\exp\left\{\frac12(O\sigma)^{\mathsf T}W(O\sigma)\right\}\\
&=\mathcal Z_N^{\mathrm{sph}}(W).
\end{align*}
The function $x\mapsto x^{-2}$ is convex on $(0,\infty)$, so conditional Jensen gives
\[
\E_O\bigl[\mathcal Z_N(O^{\mathsf T}WO)^{-2}\bigr]
\ge \bigl(\mathcal Z_N^{\mathrm{sph}}(W)\bigr)^{-2}.
\]
Taking expectation over $W$ and using the orthogonal invariance of the GOE law proves \eqref{eq:lower-haar-jensen}.
\end{proof}

\begin{proposition}\label{prop:lower-spherical-negative-moment}
There are universal constants $c>0$ and $N_0<\infty$ such that, for every $N\ge N_0$,
\begin{equation}\label{eq:lower-spherical-negative-moment}
\E\bigl(\mathcal Z_N^{\mathrm{sph}}\bigr)^{-2}
\ge cN^{1/2}e^{-N/2}.
\end{equation}
\end{proposition}

The proof of Proposition~\ref{prop:lower-spherical-negative-moment} is divided into several steps. We begin with a contour formula. Write the eigenvalues of $W$ in decreasing order as $\lambda_1>\lambda_2>\cdots>\lambda_N.$  The inequalities are strict almost surely. For $\gamma>\lambda_1$, set
\[
D_N(\gamma)=\det(\gamma I-W)=\prod_{j=1}^N(\gamma-\lambda_j).
\]

\begin{lemma}\label{lem:lower-spherical-contour}
For every real $\gamma>\lambda_1$,
\begin{equation}\label{eq:lower-spherical-contour}
\mathcal Z_N^{\mathrm{sph}}
=
\frac{\Gamma(N/2)}{2\pi i\,(N/2)^{N/2-1}}
\int_{\gamma-i\infty}^{\gamma+i\infty}
 e^{Nz/2}\det(zI-W)^{-1/2}\,dz.
\end{equation}
For each factor $(z-\lambda_j)^{-1/2}$, the branch is chosen to be positive for real $z>\lambda_1$ and analytic on $\mathbb C\setminus(-\infty,\lambda_j]$.
\end{lemma}

\begin{proof}
This contour representation is the one used in the analysis of the spherical SK model; see, for example, \cite{BaikLee2016}. We include the normalization argument. Let $d\omega$ be surface measure on the unit sphere $S^{N-1}$, let $|S^{N-1}|=2\pi^{N/2}/\Gamma(N/2)$, and let $d\bar\omega=d\omega/|S^{N-1}|$. For $r>0$, define
\[
J_W(r)=\int_{S^{N-1}}e^{r u^{\mathsf T}Wu/2}\,d\bar\omega(u).
\]
Then $J_W(N)=\mathcal Z_N^{\mathrm{sph}}$. If $z>\lambda_1$ is real, the Gaussian integral and polar coordinates give
\begin{align*}
(2\pi)^{N/2}\det(zI-W)^{-1/2}
&=\int_{\mathbb R^N}e^{-x^{\mathsf T}(zI-W)x/2}\,dx\\
&=\frac{|S^{N-1}|}{2}\int_0^\infty
 r^{N/2-1}e^{-zr/2}J_W(r)\,dr.
\end{align*}
Since $|S^{N-1}|/2=\pi^{N/2}/\Gamma(N/2)$, then we have
\begin{equation}\label{eq:lower-laplace-transform}
\int_0^\infty r^{N/2-1}e^{-zr/2}J_W(r)\,dr
=2^{N/2}\Gamma(N/2)\det(zI-W)^{-1/2}.
\end{equation}
The function $r\mapsto r^{N/2-1}J_W(r)$ is continuous on $(0,\infty)$, locally integrable at zero, and of exponential order. Hence the Bromwich inversion theorem applies at $r=N$. Applying it to \eqref{eq:lower-laplace-transform} and then changing variables from the Laplace variable $s$ to $z=2s$ gives
\begin{align*}
N^{N/2-1}J_W(N)
&=\frac{2^{N/2}\Gamma(N/2)}{2\pi i}
  \int_{\gamma/2-i\infty}^{\gamma/2+i\infty}
  e^{Ns}\det(2sI-W)^{-1/2}\,ds\\
&=\frac{2^{N/2-1}\Gamma(N/2)}{2\pi i}
  \int_{\gamma-i\infty}^{\gamma+i\infty}
  e^{Nz/2}\det(zI-W)^{-1/2}\,dz.
\end{align*}
Dividing by $N^{N/2-1}$ proves \eqref{eq:lower-spherical-contour}.
\end{proof}

We next choose a contour point at the GOE soft edge. For $n\ge1$ and $a>0$, define the Mehta integral
\begin{equation}\label{eq:lower-mehta-definition}
\mathfrak C_{n,a}
=\int_{\mathbb R^n}
 e^{-a\sum_{j=1}^n x_j^2/4}
 \prod_{1\le i<j\le n}|x_i-x_j|\,dx_1\cdots dx_n.
\end{equation}
Let $\mathbb Q_N$ be the probability law on the ordered region
$\{\mu_1>\cdots>\mu_{N+1}\}$ with density
\begin{equation}\label{eq:lower-Nplusone-density}
\frac{(N+1)!}{\mathfrak C_{N+1,N}}
 e^{-N\sum_{j=1}^{N+1}\mu_j^2/4}
 \prod_{1\le i<j\le N+1}(\mu_i-\mu_j).
\end{equation}
This is the ordered eigenvalue law of $\sqrt{(N+1)/N}$ times a standard $\mathrm{GOE}(N+1)$ matrix.

\begin{lemma}\label{lem:lower-edge-point}
There are constants $q_-<q_+$, $L>-q_-$, $c_0>0$, and, for every sufficiently large $N$, a deterministic point $\gamma_N\in \mathcal{I}_N
=\bigl(2+q_-N^{-2/3},2+q_+N^{-2/3}\bigr)$ such that, with $b_N=2-LN^{-2/3}$,
\begin{equation}\label{eq:lower-edge-density}
r_N(\gamma_N)\ge c_0N^{2/3}.
\end{equation}
Here $r_N$ is the density of the finite measure $\mathbb Q_N(\mu_1\in dx,\ \mu_4\ge b_N)=r_N(x)\,dx.$ Moreover,
\begin{equation}\label{eq:lower-edge-distance}
0<\gamma_N-b_N\le(q_++L)N^{-2/3}.
\end{equation}
\end{lemma}

\begin{proof}
The joint soft edge convergence for the Gaussian beta ensembles implies that the first four ordered eigenvalues under $\mathbb Q_N$, centered at $2$ and scaled by $N^{2/3}$, converge jointly to the first four points of the GOE stochastic Airy spectrum, up to a fixed positive deterministic scaling constant; see \cite{RamirezRiderVirag2011}. The factor $\sqrt{(N+1)/N}=1+O(N^{-1})$ changes the edge by $O(N^{-1})$, which is smaller than the $N^{-2/3}$ scale and therefore does not affect the conclusion.

Choose finite numbers $q_-<q_+$ and $L>-q_-$, at continuity points of the limiting joint distribution, so that the limiting event corresponding to $\mu_1\in \mathcal{I}_N,$ and $\mu_4\ge b_N$ has probability $2p_0>0$. Joint convergence then gives, for all sufficiently large $N$,
\begin{equation}\label{eq:lower-edge-positive-mass}
\mathbb Q_N(\mu_1\in \mathcal{I}_N,\ \mu_4\ge b_N)\ge p_0.
\end{equation}
The density $r_N$ can be written  by integrating \eqref{eq:lower-Nplusone-density} over $\mu_2,\ldots,\mu_{N+1}$, and hence
\[
\int_{I_N}r_N(x)\,dx
=\mathbb Q_N(\mu_1\in I_N,\ \mu_4\ge b_N).
\]
Since $|\mathcal{I}_N|=(q_+-q_-)N^{-2/3}$, \eqref{eq:lower-edge-positive-mass} implies that there is a deterministic $\gamma_N\in \mathcal{I}_N$ for which
\[
r_N(\gamma_N)
\ge\frac{p_0}{|\mathcal{I}_N|}
=\frac{p_0}{q_+-q_-}N^{2/3}.
\]
This proves \eqref{eq:lower-edge-density}. Finally, $L>-q_-$ implies $b_N<\inf \mathcal{I}_N$, and $\gamma_N\le2+q_+N^{-2/3}$ gives \eqref{eq:lower-edge-distance}.
\end{proof}

For the $N\times N$ matrix $W$, define the event
\begin{equation}\label{eq:lower-edge-event}
\mathcal A_N=\{\lambda_1<\gamma_N,\ \lambda_3\ge b_N\}.
\end{equation}
On $\mathcal A_N$, put $a_j=\gamma_N-\lambda_j$. Then $a_j>0$ for every $j$, and for $j=1,2,3$,
\begin{equation}\label{eq:lower-first-three-gaps}
a_j\le\gamma_N-b_N\le(q_++L)N^{-2/3}.
\end{equation}
Parametrize the contour in \eqref{eq:lower-spherical-contour} by $z=\gamma_N+it$. Taking absolute values gives
\begin{align}
\mathcal Z_N^{\mathrm{sph}}
&\le c_N e^{N\gamma_N/2}D_N(\gamma_N)^{-1/2}
\int_{\mathbb R}\prod_{j=1}^N
\left(1+\frac{t^2}{a_j^2}\right)^{-1/4}\,dt,
\label{eq:lower-contour-absolute}
\end{align}
where $c_N=\frac{\Gamma(N/2)}{2\pi(N/2)^{N/2-1}}.$ Indeed,
\[
|\gamma_N+it-\lambda_j|^{-1/2}
=a_j^{-1/2}\left(1+\frac{t^2}{a_j^2}\right)^{-1/4}.
\]
Stirling's formula yields
\begin{equation}\label{eq:lower-contour-prefactor}
c_N=\frac{\sqrt N}{2\sqrt\pi}e^{-N/2}\bigl(1+O(N^{-1})\bigr),
\end{equation}
and in particular $c_N\le C\sqrt N e^{-N/2}$.

Let $r=(q_++L)N^{-2/3}$. By \eqref{eq:lower-first-three-gaps}, for every $t\in\mathbb R$ and on $\mathcal A_N$,
\begin{align*}
\prod_{j=1}^N
\left(1+\frac{t^2}{a_j^2}\right)^{-1/4}
&\le
\prod_{j=1}^3
\left(1+\frac{t^2}{a_j^2}\right)^{-1/4}\\
&\le
\left(1+\frac{t^2}{r^2}\right)^{-3/4}.
\end{align*}
The exponent $3/4$ is larger than $1/2$, so the resulting integral is finite, and
\begin{equation}\label{eq:lower-contour-integral-bound}
\int_{\mathbb R}\prod_{j=1}^N
\left(1+\frac{t^2}{a_j^2}\right)^{-1/4}\,dt
\le
r\int_{\mathbb R}(1+u^2)^{-3/4}\,du
\le CN^{-2/3}.
\end{equation}
Combining \eqref{eq:lower-contour-absolute}, \eqref{eq:lower-contour-prefactor}, and \eqref{eq:lower-contour-integral-bound}, we obtain on $\mathcal A_N$
\begin{equation}\label{eq:lower-spherical-upper-on-event}
\mathcal Z_N^{\mathrm{sph}}
\le
CN^{-1/6}e^{-N/2+N\gamma_N/2}D_N(\gamma_N)^{-1/2}.
\end{equation}
Therefore,
\begin{equation}\label{eq:lower-spherical-inverse-on-event}
\bigl(\mathcal Z_N^{\mathrm{sph}}\bigr)^{-2}
\ge
cN^{1/3}e^{N-N\gamma_N}D_N(\gamma_N)\,\1_{\mathcal A_N}.
\end{equation}

The determinant in \eqref{eq:lower-spherical-inverse-on-event} can be absorbed into the Vandermonde factor by adding one eigenvalue.

\begin{lemma}\label{lem:lower-eigenvalue-insertion}
For every $x>b_N$,
\begin{equation}\label{eq:lower-eigenvalue-insertion}
\E\left[D_N(x)\1_{\{\lambda_1<x,\ \lambda_3\ge b_N\}}\right]
=
\frac{\mathfrak C_{N+1,N}}{(N+1)\mathfrak C_{N,N}}
 e^{Nx^2/4}r_N(x).
\end{equation}
\end{lemma}

\begin{proof}
On the ordered region $\lambda_1>\cdots>\lambda_N$, the eigenvalue density of $W$ is
\[
\frac{N!}{\mathfrak C_{N,N}}
 e^{-N\sum_{j=1}^N\lambda_j^2/4}
 \prod_{1\le i<j\le N}(\lambda_i-\lambda_j).
\]
For brevity, write $d\lambda=d\lambda_1\cdots d\lambda_N$. If $x>\lambda_1$, then
\begin{align*}
D_N(x)\prod_{1\le i<j\le N}(\lambda_i-\lambda_j)
&=\prod_{j=1}^N(x-\lambda_j)
  \prod_{1\le i<j\le N}(\lambda_i-\lambda_j)\\
&=\prod_{1\le i<j\le N+1}(\mu_i-\mu_j),
\end{align*}
where $\mu_1=x$ and $\mu_{j+1}=\lambda_j$ for $1\le j\le N$. Hence
\begin{align}
&\E\left[D_N(x)\1_{\{\lambda_1<x,\ \lambda_3\ge b_N\}}\right]\notag\\
&\quad=
\frac{N!}{\mathfrak C_{N,N}}
\int_{\substack{x>\lambda_1>\cdots>\lambda_N\\ \lambda_3\ge b_N}}
 e^{-N\sum_{j=1}^N\lambda_j^2/4}
 \prod_{j=1}^N(x-\lambda_j)\notag\\
&\qquad\quad\times
 \prod_{1\le i<j\le N}(\lambda_i-\lambda_j)\,d\lambda.
\label{eq:lower-insertion-left-integral}
\end{align}
On the other hand, fixing the largest coordinate $\mu_1=x$ in the ordered density \eqref{eq:lower-Nplusone-density} gives
\begin{align}
r_N(x)
&=
\frac{(N+1)!}{\mathfrak C_{N+1,N}}e^{-Nx^2/4}
\int_{\substack{x>\lambda_1>\cdots>\lambda_N\\ \lambda_3\ge b_N}}
 e^{-N\sum_{j=1}^N\lambda_j^2/4}
 \prod_{j=1}^N(x-\lambda_j)\notag\\
&\qquad\quad\times
 \prod_{1\le i<j\le N}(\lambda_i-\lambda_j)\,d\lambda.
\label{eq:lower-insertion-right-integral}
\end{align}
The condition $\mu_4\ge b_N$ becomes $\lambda_3\ge b_N$. Comparing \eqref{eq:lower-insertion-left-integral} and \eqref{eq:lower-insertion-right-integral}, and using $N!/(N+1)!=1/(N+1)$, proves \eqref{eq:lower-eigenvalue-insertion}.
\end{proof}

We now consider the normalizing constant ratio. The Mehta integral formula \cite{Forrester2010} states that
\begin{equation}\label{eq:lower-mehta-formula}
\mathfrak C_{n,a}
=
\left(\frac2a\right)^{n(n+1)/4}
(2\pi)^{n/2}
\prod_{j=1}^n\frac{\Gamma(1+j/2)}{\Gamma(3/2)}.
\end{equation}
Applying \eqref{eq:lower-mehta-formula} with $(n,a)=(N+1,N)$ and $(N,N)$ gives
\begin{equation}\label{eq:lower-mehta-ratio-exact}
\frac{\mathfrak C_{N+1,N}}{(N+1)\mathfrak C_{N,N}}
=
\frac{\sqrt{2\pi}}{N+1}
\left(\frac2N\right)^{(N+1)/2}
\frac{\Gamma((N+3)/2)}{\Gamma(3/2)}.
\end{equation}
Set $m=(N+1)/2$. Stirling's formula gives
\[
\Gamma(m+1)=\sqrt{2\pi m}\left(\frac me\right)^m
\bigl(1+O(N^{-1})\bigr).
\]
Moreover,
\[
\left(\frac2N\right)^m m^m
=\left(1+\frac1N\right)^{(N+1)/2}
=e^{1/2}\bigl(1+O(N^{-1})\bigr)
\]
and $e^{-m}=e^{-N/2-1/2}$. Since $\Gamma(3/2)=\sqrt\pi/2$, substitution into \eqref{eq:lower-mehta-ratio-exact} yields
\begin{equation}\label{eq:lower-mehta-ratio-asymptotic}
\frac{\mathfrak C_{N+1,N}}{(N+1)\mathfrak C_{N,N}}
=\sqrt{8\pi}\,N^{-1/2}e^{-N/2}
\bigl(1+O(N^{-1})\bigr).
\end{equation}
In particular, the ratio in \eqref{eq:lower-mehta-ratio-asymptotic} is at least $cN^{-1/2}e^{-N/2}$ for all sufficiently large $N$.

We can now finish the proof of the spherical negative moment lower bound.

\begin{proof}[Proof of Proposition~\ref{prop:lower-spherical-negative-moment}]
Taking expectation in \eqref{eq:lower-spherical-inverse-on-event} gives
\begin{equation}\label{eq:lower-spherical-moment-before-insertion}
\E\bigl(\mathcal Z_N^{\mathrm{sph}}\bigr)^{-2}
\ge
cN^{1/3}e^{N-N\gamma_N}
\E\left[D_N(\gamma_N)\1_{\mathcal A_N}\right].
\end{equation}
By Lemma~\ref{lem:lower-eigenvalue-insertion}, \eqref{eq:lower-edge-density}, and \eqref{eq:lower-mehta-ratio-asymptotic},
\begin{align}
\E\left[D_N(\gamma_N)\1_{\mathcal A_N}\right]
&=
\frac{\mathfrak C_{N+1,N}}{(N+1)\mathfrak C_{N,N}}
 e^{N\gamma_N^2/4}r_N(\gamma_N)\notag\\
&\ge
cN^{-1/2}e^{-N/2}
 e^{N\gamma_N^2/4}N^{2/3}\\
&=
 cN^{1/6}e^{-N/2+N\gamma_N^2/4}.
\label{eq:lower-determinant-expectation}
\end{align}
Substituting \eqref{eq:lower-determinant-expectation} into \eqref{eq:lower-spherical-moment-before-insertion}, we obtain
\begin{align*}
\E\bigl(\mathcal Z_N^{\mathrm{sph}}\bigr)^{-2}
&\ge
cN^{1/2}
\exp\left\{
N-N\gamma_N-\frac N2+\frac{N\gamma_N^2}{4}
\right\}\\
&=
 cN^{1/2}
\exp\left\{-\frac N2+\frac N4(\gamma_N-2)^2\right\}\\
&\ge cN^{1/2}e^{-N/2}.
\end{align*}
\end{proof}

\begin{proof}[Proof of Theorem~\ref{thm:critical-lower-bound}]
By Lemma~\ref{lem:lower-haar-jensen} and Proposition~\ref{prop:lower-spherical-negative-moment},
\begin{equation}\label{eq:lower-cube-negative-moment}
\E\mathcal Z_N^{-2}\ge cN^{1/2}e^{-N/2}
\end{equation}
for all sufficiently large $N$. Therefore,
\begin{equation}
\frac12\log\E\mathcal Z_N^{-2}
\ge
-\frac N4+\frac14\log N-C.
\label{eq:Z_N_bound}
\end{equation}
Combining \eqref{eq:Z_N_bound} with \eqref{eq:lower-convexity-reduction} and \eqref{eq:lower-entropy-input}, we find
\begin{align*}
\Var(\mathcal F_N)
&\ge
\frac N4-\frac N4+\frac14\log N-B_N(\beta_c)-C\\
&\ge
\left(\frac14-\frac1{12}\right)\log N-C\\
&=\frac16\log N-C.
\end{align*}
Equation~\eqref{eq:lower-full-to-goe} gives the same lower bound for $F_N(\beta_c)$. Increasing $C$ handles the finitely many values $2\le N<N_0$.
\end{proof}

\section*{Acknowledgements}
The author would like to express his sincere gratitude to Professor Jian Ding and Dr.\ Hang Du for valuable discussions and insightful suggestions, and to Zijun Chen for carefully checking the proofs at an early stage of this work.
 The author is supported by the National Natural Science Foundation of China (Grant Nos.\ 12595294 and 12231002) and the New Cornerstone Science Foundation (Grant No.\ NCI202501).
%

\bibliographystyle{plain}
\bibliography{ref}

@article{Asp2008,
  AUTHOR  = {Aspelmeier, Timo},
  TITLE   = {Free-Energy Fluctuations and Chaos in the Sherrington-Kirkpatrick Model},
  JOURNAL = {Phys. Rev. Lett.},
 FJOURNAL = {Physical Review Letters},
  VOLUME  = {100},
  NUMBER  = {11},
  PAGES   = {117205},
  YEAR    = {2008},
  DOI     = {10.1103/PhysRevLett.100.117205},
  URL      ={https://link.aps.org/doi/10.1103/PhysRevLett.100.117205}
}

@article {CL2019,
    AUTHOR = {Chen, Wei-Kuo and Lam, Wai-Kit},
     TITLE = {Order of fluctuations of the free energy in the {SK} model at
              critical temperature},
   JOURNAL = {ALEA Lat. Am. J. Probab. Math. Stat.},
  FJOURNAL = {ALEA. Latin American Journal of Probability and Mathematical
              Statistics},
    VOLUME = {16},
      YEAR = {2019},
    NUMBER = {1},
     PAGES = {809--816},
      DOI = {10.30757/alea.v16-29},
       URL = {https://doi.org/10.30757/alea.v16-29},
}

@article {Chen2024GaussianConvexity,
    AUTHOR = {Chen, Wei-Kuo},
     TITLE = {A {G}aussian convexity for logarithmic moment generating
              functions with applications in spin glasses},
   JOURNAL = {Ann. Inst. Henri Poincar\'e{} Probab. Stat.},
  FJOURNAL = {Annales de l'Institut Henri Poincar\'e{} Probabilit\'es et
              Statistiques},
    VOLUME = {62},
      YEAR = {2026},
    NUMBER = {1},
      DOI = {10.1214/24-aihp1527},
       URL = {https://doi.org/10.1214/24-aihp1527},
}

@article{DK2026,
  title={Fluctuations for the Sherrington--Kirkpatrick spin glass model near the critical temperature},
  author={Dey, Partha S and Kang, Taegu},
  journal={arXiv preprint arXiv:2603.05636},
  year={2026}
}

@article {DEZ2015,
    AUTHOR = {Ding, Jian and Eldan, Ronen and Zhai, Alex},
     TITLE = {On multiple peaks and moderate deviations for the supremum of
              a {G}aussian field},
   JOURNAL = {Ann. Probab.},
  FJOURNAL = {The Annals of Probability},
    VOLUME = {43},
      YEAR = {2015},
    NUMBER = {6},
     PAGES = {3468--3493},
       DOI = {10.1214/14-AOP963},
       URL = {https://doi.org/10.1214/14-AOP963},
}

@article {Ehr1983,
    AUTHOR = {Ehrhard, Antoine},
     TITLE = {Sym\'etrisation dans l'espace de {G}auss},
   JOURNAL = {Math. Scand.},
  FJOURNAL = {Mathematica Scandinavica},
    VOLUME = {53},
      YEAR = {1983},
    NUMBER = {2},
     PAGES = {281--301},
      ISSN = {0025-5521,1903-1807},
   MRCLASS = {60G15 (35K05 52A22 60B99)},
  MRNUMBER = {745081},
MRREVIEWER = {A.\ Badrikian},
       DOI = {10.7146/math.scand.a-12035},
       URL = {https://doi.org/10.7146/math.scand.a-12035},
}

@book{MPV1987,
  author    = {M{\'e}zard, Marc and Parisi, Giorgio and Virasoro, Miguel A.},
  title     = {Spin Glass Theory and Beyond: An Introduction to the Replica Method and Its Applications},
  series    = {World Scientific Lecture Notes in Physics},
  volume    = {9},
  publisher = {World Scientific Publishing Co., Inc.},
  address   = {Teaneck, NJ},
  year      = {1987}
}

@book{Pan2013,
  author    = {Panchenko, Dmitry},
  title     = {The Sherrington--Kirkpatrick Model},
  series    = {Springer Monographs in Mathematics},
  publisher = {Springer},
  address   = {New York, NY},
  year      = {2013}
}

@article{PR2009,
  AUTHOR  = {Parisi, Giorgio and Rizzo, Tommaso},
  TITLE   = {Phase diagram and large deviations in the free energy of mean-field spin-glasses},
  JOURNAL = {Phys. Rev. B},
  FJOURNAL = {Physical Review B},
  VOLUME  = {79},
  NUMBER  = {13},
  PAGES   = {134205},
  YEAR    = {2009},
  DOI    = {10.1103/PhysRevB.79.134205},
  URL     = {https://link.aps.org/doi/10.1103/PhysRevB.79.134205}
}

@article{SK1975,
  AUTHOR  = {Sherrington, David and Kirkpatrick, Scott},
  TITLE   = {Solvable Model of a Spin-Glass},
  JOURNAL = {Phys. Rev. Lett.},
 FJOURNAL = {Physical Review Letters},
  VOLUME  = {35},
  NUMBER  = {26},
  PAGES   = {1792--1796},
  YEAR    = {1975},
  DOI    = {10.1103/PhysRevLett.35.1792},
  URL    = {https://link.aps.org/doi/10.1103/PhysRevLett.35.1792}
}

@book{Tal2003,
  author    = {Talagrand, Michel},
  title     = {Spin Glasses: A Challenge for Mathematicians: Cavity and Mean Field Models},
  series    = {Ergebnisse der Mathematik und ihrer Grenzgebiete. 3. Folge. A Series of Modern Surveys in Mathematics},
  volume    = {46},
  publisher = {Springer},
  address   = {Berlin, Heidelberg},
  year      = {2003}
}

@book{Tal2011spinI,
  author    = {Talagrand, Michel},
  title     = {Mean Field Models for Spin Glasses. Volume I: Basic Examples},
  series    = {Ergebnisse der Mathematik und ihrer Grenzgebiete. 3. Folge. A Series of Modern Surveys in Mathematics},
  volume    = {54},
  publisher = {Springer},
  address   = {Berlin, Heidelberg},
  year      = {2011}
}

@book{Tal2011spinII,
  author    = {Talagrand, Michel},
  title     = {Mean Field Models for Spin Glasses. Volume II: Advanced Replica-Symmetry and Low Temperature},
  series    = {Ergebnisse der Mathematik und ihrer Grenzgebiete. 3. Folge. A Series of Modern Surveys in Mathematics},
  volume    = {55},
  publisher = {Springer},
    address   = {Berlin, Heidelberg},
  year      = {2011}
}

@book {VER2020,
    AUTHOR = {Vershynin, Roman},
     TITLE = {High-dimensional probability},
    SERIES = {Cambridge Series in Statistical and Probabilistic Mathematics},
    VOLUME = {47},
      NOTE = {An introduction with applications in data science,
              With a foreword by Sara van de Geer},
 PUBLISHER = {Cambridge University Press, Cambridge},
      YEAR = {2018},
     PAGES = {xiv+284},
       DOI = {10.1017/9781108231596},
       URL = {https://doi.org/10.1017/9781108231596},
}

@article{AbbeBoix2020,
AUTHOR  = {Abbe, Emmanuel and Boix-Adser{`a}, Enric},
TITLE   = {An Information-Percolation Bound for Spin Synchronization on General Graphs},
JOURNAL = {Ann. Appl. Probab.},
FJOURNAL = {The Annals of Applied Probability},
VOLUME  = {30},
NUMBER  = {3},
PAGES   = {1066--1090},
YEAR    = {2020},
DOI     = {10.1214/19-AAP1523},
URL     = {https://doi.org/10.1214/19-AAP1523}
}

@article{GuoShamaiVerdu2005,
AUTHOR  = {Guo, Dongning and Shamai, Shlomo and Verd{'u}, Sergio},
TITLE   = {Mutual Information and Minimum Mean-Square Error in Gaussian Channels},
JOURNAL = {IEEE Trans. Inform. Theory},
FJOURNAL = {IEEE Transactions on Information Theory},
VOLUME  = {51},
NUMBER  = {4},
PAGES   = {1261--1282},
YEAR    = {2005},
DOI     = {10.1109/TIT.2005.844072},
URL     = {https://doi.org/10.1109/TIT.2005.844072}
}

@article{JansonSpencer2007,
AUTHOR  = {Janson, Svante and Spencer, Joel},
TITLE   = {A Point Process Describing the Component Sizes in the Critical Window of the Random Graph Evolution},
JOURNAL = {Combin. Probab. Comput.},
FJOURNAL = {Combinatorics, Probability and Computing},
VOLUME  = {16},
NUMBER  = {4},
PAGES   = {631--658},
YEAR    = {2007},
DOI     = {10.1017/S0963548306008327},
URL     = {https://doi.org/10.1017/S0963548306008327}
}

@article{Schertzer2026,
AUTHOR  = {Schertzer, Adrien},
TITLE   = {The Order of Free Energy Fluctuations in the Critical {Sherrington--Kirkpatrick} Model Revisited},
JOURNAL = {arXiv preprint arXiv:2606.21360},
FJOURNAL = {arXiv preprint},
YEAR    = {2026},
EPRINT  = {2606.21360},
ARCHIVEPREFIX = {arXiv},
PRIMARYCLASS  = {math.PR},
URL     = {https://arxiv.org/abs/2606.21360}
}

@article{BaikLee2016,
  AUTHOR  = {Baik, Jinho and Lee, Ji Oon},
  TITLE   = {Fluctuations of the Free Energy of the Spherical Sherrington--Kirkpatrick Model},
  JOURNAL = {J. Stat. Phys.},
  FJOURNAL = {Journal of Statistical Physics},
  VOLUME  = {165},
  NUMBER  = {2},
  PAGES   = {185--224},
  YEAR    = {2016},
  DOI     = {10.1007/s10955-016-1610-0},
  URL     = {https://doi.org/10.1007/s10955-016-1610-0}
}

@article{RamirezRiderVirag2011,
  AUTHOR  = {Ram{\'i}rez, Jos{\'e} A. and Rider, Brian and Vir{\'a}g, B{\'a}lint},
  TITLE   = {Beta Ensembles, Stochastic Airy Spectrum, and a Diffusion},
  JOURNAL = {J. Amer. Math. Soc.},
  FJOURNAL = {Journal of the American Mathematical Society},
  VOLUME  = {24},
  NUMBER  = {4},
  PAGES   = {919--944},
  YEAR    = {2011},
  DOI     = {10.1090/S0894-0347-2011-00703-0},
  URL     = {https://doi.org/10.1090/S0894-0347-2011-00703-0}
}

@book{Forrester2010,
  AUTHOR    = {Forrester, Peter J.},
  TITLE     = {Log-Gases and Random Matrices},
  SERIES    = {London Mathematical Society Monographs Series},
  VOLUME    = {34},
  PUBLISHER = {Princeton University Press},
  ADDRESS   = {Princeton, NJ},
  YEAR      = {2010},
  DOI       = {10.1515/9781400835416},
  URL       = {https://doi.org/10.1515/9781400835416}
}

@article{DuHuang2026,
  AUTHOR  = {Du, Hang and Huang, Brice},
  TITLE   = {Fluctuations of the {Sherrington--Kirkpatrick} Free Energy at Critical Temperature},
  JOURNAL = {arXiv preprint arXiv:2607.02172},
  YEAR    = {2026},
  EPRINT  = {2607.02172},
  ARCHIVEPREFIX = {arXiv},
  PRIMARYCLASS  = {math.PR},
  URL     = {https://arxiv.org/abs/2607.02172}
}

@article{Borell2003,
  AUTHOR  = {Borell, Christer},
  TITLE   = {The Ehrhard Inequality},
  JOURNAL = {C. R. Math. Acad. Sci. Paris},
  FJOURNAL = {Comptes Rendus Math{\'e}matique. Acad{\'e}mie des Sciences. Paris},
  VOLUME  = {337},
  NUMBER  = {10},
  PAGES   = {663--666},
  YEAR    = {2003},
  DOI     = {10.1016/j.crma.2003.09.031},
  URL     = {https://doi.org/10.1016/j.crma.2003.09.031}
}

@article{PaourisValettas2018,
  AUTHOR  = {Paouris, Grigoris and Valettas, Petros},
  TITLE   = {A Gaussian Small Deviation Inequality for Convex Functions},
  JOURNAL = {Ann. Probab.},
  FJOURNAL = {The Annals of Probability},
  VOLUME  = {46},
  NUMBER  = {3},
  PAGES   = {1441--1454},
  YEAR    = {2018},
  DOI     = {10.1214/17-AOP1206},
  URL     = {https://doi.org/10.1214/17-AOP1206}
}

@article{Panchenko2014,
  AUTHOR  = {Panchenko, Dmitry},
  TITLE   = {The Parisi Formula for Mixed \(p\)-Spin Models},
  JOURNAL = {Ann. Probab.},
  FJOURNAL = {The Annals of Probability},
  VOLUME  = {42},
  NUMBER  = {3},
  PAGES   = {946--958},
  YEAR    = {2014},
  DOI     = {10.1214/12-AOP800},
  URL     = {https://doi.org/10.1214/12-AOP800}
}

@article{CarmonaHu2002,
  AUTHOR  = {Carmona, Philippe and Hu, Yueyun},
  TITLE   = {On the Partition Function of a Directed Polymer in a Gaussian Random Environment},
  JOURNAL = {Probab. Theory Related Fields},
  FJOURNAL = {Probability Theory and Related Fields},
  VOLUME  = {124},
  NUMBER  = {3},
  PAGES   = {431--457},
  YEAR    = {2002},
  DOI     = {10.1007/s004400200213},
  URL     = {https://doi.org/10.1007/s004400200213}
}

\end{document}